\documentclass[a4paper]{article}
\usepackage{amsmath,amsthm,amssymb,graphicx,multirow,booktabs}
\usepackage[left=3cm, right=2cm, top=1.5cm, bottom=2.2cm]{geometry}
\usepackage{algorithm}
\usepackage{algorithmic}
\usepackage{longtable,color,float}
\usepackage{multirow}
\usepackage[unicode]{hyperref}
\providecommand{\keywords}[1]{\textit{Keywords:} #1}
\usepackage[numbers,sort&compress]{natbib}

\newtheorem{theorem}{Theorem}

\newtheorem{proposition}[theorem]{Proposition}
\newtheorem{lemma}[theorem]{Lemma}

\newtheorem{defn}[theorem]{Definition} 

\begin{document}
\title{A scaled conjugate gradient based direct search algorithm for high dimensional box constrained derivative free optimization}

\author{Gannavarapu Chandramouli\\
Industrial Engineering and Operations Research, IIT Bombay Mumbai 400076, India\\
\and
Vishnu Narayanan\\
Industrial Engineering and Operations Research, IIT Bombay Mumbai 400076, India\\
}
\maketitle

	\begin{abstract}
		In this work, we propose an efficient method for solving box constrained derivative free optimization problems involving high dimensions.
		The proposed method relies on exploring the feasible region using a direct search approach based on scaled conjugate gradient with quadratic interpolation models.
		The extensive numerical computations carried out over test problems with varying dimensions demonstrate the performance of the proposed method for derivative free optimization.
	\end{abstract}

\section{Introduction}
Many real world optimization problems are formulated as
$$\underset{x \in \Omega}{\text{min} }f(x)$$
where $f:\mathbb{R}^n \rightarrow \mathbb{R}$ is a nonlinear function, and $\Omega=\{x\in \mathbb{R}^n:l_i \leq x_i \leq u_i, \; i=1,\ldots,n\}$.
In many cases, the gradient information required to solve the optimization problem is either unavailable or too expensive to compute using the standard approaches.
Such problems are categorized under the purview of box constrained \emph{derivative free optimization}(DFO), and usually arise due to complex simulations or physical experiments where cost of computing the function value is fairly high. 
The main objective of the algorithms applied for solving these problems is to obtain the optimum while being highly frugal over the number of function evaluations. 
Hence, gradient based approaches using finite differences or automatic differentiation and metaheuristic methods are often ineffective in this domain because of high cost of function evaluations. 
Further, the presence of noise or nonsmooth function nature poses additional limitations.

Two fundamental classes of algorithms \cite{connBook,audet2017derivative} for solving derivative free optimization problems with guaranteed convergence to local optimum or Clarke-Jahn stationary point \cite{clarke1990optimization} (for nonsmooth cases) include the direct search and trust region methods.
Direct-search methods \cite{connBook} explore the feasible search space by generating new points satisfying some conditions like well poisedness.
These conditions are needed to ensure the proper exploration of search space around the current iterate.
These methods generally comprise of three steps: a search step to explore the space for better solutions, a poll step to ensure convergence to some stationary point, and subsequent parameter update.
Mesh adaptive direct search (MADS) \cite{audet_dennis_jr_mesh, abramson_audet_ea_orthomads} is a well known direct search approach. 
MADS was designed to address derivative free optimization problems which have nonsmooth nature to a large extent, and sometimes, have hidden constraints.
It generates a set of directions which, when rounded to hypothetical mesh (integer lattice), are dense on unit sphere.
For handling constraints, the application of the extreme barrier approach \cite{audet_dennis_jr_mesh} has been suggested.
\citet{vicente2012analysis} proposed a replacement for the integrality requirements of direct search methods like MADS and suggested a condition of sufficient decrease in the function value during the poll step.
This line of approach was further extended \cite{fasano2014linesearch} to handle constraints using projected line search and penalty approach.
Since these algorithms do not rely on the function nature, they are quite robust and have the additional advantage of parallel implementation \cite{gray2006algorithm, Le09b,audet2008parallel}. 
However, because of their inability to explore and utilize the curvature information, they often require large number of function evaluations for convergence.
Further, direct search methods \cite{custodio2007using,custodio2008using} using simplex gradients have been proposed, but their use was limited to reordering the poll directions instead of enhancing the search step.

The other prominent method for derivative free optimization involves the trust region approach where the curvature information about the function is utilized by building models over the explored points.
In this approach, a chosen model is fitted across already evaluated points at each iteration to obtain a close approximation to the original function.
The ability to capture approximate curvature information in trust region approach is attributed for providing fast solutions to derivative free optimization problems. 
Several algorithms utilizing models based on quadratic \cite{powell2006newuoa,wang2016conjugate}, kriging \cite{jones1998efficient,gramacy_le_digabel_mesh} and radial basis function \cite{wild_regis_ea_orbit,regis2013combining} have been proposed.
Notably, these algorithms require less function evaluations than direct search methods when the underlying function is well behaved.
The effectiveness of these algorithms for solving smooth, piecewise-smooth and noisy problems in derivative free optimization is outlined in \cite{more_wild_benchmarking}.
These approaches, however, are sequential in nature and their performance is affected when the function lacks good structural properties.
Recently, approaches \cite{custodio2010incorporating,conn_le_digabel_use,amaioua2018efficient} based on the blending of quadratic models with direct search methods have been reported which utilize the curvature information of the function to find a better solution.
However, when the function lacks a good structure, these methods revert to direct search.

Many optimization problems in practice involve large number of variables, and solving them under derivative free conditions poses significant challenges.
Such problems usually have a large number of local optima and lack good structural properties.
Consequently, finding global optimum for these problems is generally unrealistic and non-trivial because of substantial number of function evaluations needed.
Also, these problems are often plagued with noise (stochastic or non-stochastic) which generally leads to high and low frequency oscillations \cite{more_wild_benchmarking}; for instance, solving a differential equation with multiple parameters to a specified accuracy.
Further, presence of even simple constraints like bound ones, introduces additional limitations and increases the overall complexity of the problem. 
Thus, a good local optimization algorithm which can provide sufficiently good solutions with a small computational budget is desirable.
Further, the required algorithm should have the ability to exploit certain structural properties such as convexity and smoothness, wherever possible, in order to enhance its speed and accuracy.

In this work, we propose an elegant approach based on direct search for solving box constrained derivative free optimization problems.
In the proposed approach, we suggest a new strategy for integrating the quadratic models into direct search framework to achieve high performance due to the synergy of curvature information retrieval ability of quadratic models with search directions provided by scaled conjugate gradient. 
The structure of the paper is as follows. 
In section \ref{nota}, we give an overview of few definitions required for proving convergence results.
In section \ref{Background}, we provide background information about direct search, quadratic models, simplex gradient and scaled conjugate gradient.
In section \ref{Theory}, we outline the proposed approach and discuss convergence proofs.
We report our numerical results in section \ref{Results} followed by conclusions in section \ref{Conclusion}.

\section{Notation and definitions}
\label{nota}
For vectors $u,v\in \mathbb{R}^n$, we define max$\{u,v\}=x$ and min$\{u,v\}=y$ where $x,y \in \mathbb{R}^n$ such that $x_i=\text{max}\{u_i,v_i\}$ and $y_i=\text{min}\{u_i,v_i\}$ for $i=1,2,\ldots,n$. 
Let $\Omega=\{x\in \mathbb{R}^n:l_i \leq x_i \leq u_i,\; i=1,\ldots,n\}$.
We denote a ball with center $c \in \mathbb{R}^n$ and radius $r \in \mathbb{R}_+$ as $\mathcal{B}(c,r)$.
We now present few definitions and lemmas from \cite{fasano2014linesearch}, which are required for proving convergence results.
\begin{defn}
	For a point $x \in \Omega$, cone of feasible directions $D(x)$ is defined as:
	\begin{equation*}
	D(x)=\{d\in \mathbb{R}^n:\, d_i \geq 0 \text{ if } x_i=l_i,\; d_i \leq 0 \text{ if } x_i=u_i,\; d_i \in \mathbb{R} \text{ if } l_i<x_i<u_i, \; i=1,\ldots,n\}
	\end{equation*}
\end{defn}

\begin{lemma}\label{coneLemma}
	Let $\{x_k\} \subset \Omega$ be a sequence such that $\underset{k \rightarrow \infty}{\text{lim}}x_k \rightarrow \bar{x}$. There exists $m \in \mathbb{N}$ such that for all $k>m$,
	\begin{equation*}
		D(\bar{x}) \subseteq D(x_k).
	\end{equation*}
\end{lemma}

\begin{defn}
	The Clarke-Jahn generalized directional derivative (\cite{clarke1990optimization}) of Lipschitz function $f$ at $x \in \Omega$ along direction $d \in D(x)$ is defined as 
\begin{equation}
	f^{\circ}(x;d)=\underset{y\rightarrow x,\; t \downarrow 0}{\text{lim sup }}\frac{f(y+td)-f(y)}{t},
\end{equation}
	where $y\in \Omega$ and $y+td \in \Omega$. Also, $\hat{x} \in \Omega$ is a Clarke-Jahn stationary point for $f$ if
	\begin{equation}
		f^{\circ}(\hat{x};d) \geq 0 \qquad \forall d \in D(\hat{x}).
\end{equation}
\end{defn}

\begin{defn}
	Let $D=\{d_k\}_K$ be a sequence of normalized directions where $K$ is a set of indices i.e. $K=\{0,1,\ldots\}$. Given any direction $d \in \mathbb{R}^n$ such that $\|d\|_2=1$ i.e. $d \in \mathcal{B}(0,1)$ and any $\epsilon >0$, if there exists a direction $d_k \in D$ such that $\|d-d_k\|_2 < \epsilon$, then $D$ is said to be dense \cite{fasano2014linesearch} in $\mathcal{B}(0,1)$.
\end{defn}

\section{Background}\label{Background}
\subsection{Direct Search}
A typical direct search iteration is composed of following steps:
\begin{description}
\item[Search Step:] In this step trial points are generated using some user defined approach. 
Implementation of problem specific procedures in the search step can greatly enhance the performance of the algorithm.
Literature suggests the use of approaches like variable neighbourhood search \cite{audet_bechard_ea_nonsmooth}, quadratic models \cite{conn_le_digabel_use}, and Treed Gaussian Process (TGP) \cite{gramacy_le_digabel_mesh} in the search step of MADS.
If a new point with function value better than incumbent solution is found, the search step is declared as successful.

\item[Poll Step:] Unlike the search step, this does not offer any flexibility, but ensures the convergence to a local optimum.
In this step, finite number of new directions are constructed to create new trial points at each iteration.
		These directions are in fact vectors of some positive basis (different at each iteration) and the distance of generated trial points is bounded above by a poll size parameter.
The set of such directions generated over infinite iterations, when normalized, are required to be dense on the unit sphere.
If any of the trial points along these directions has function value better than the incumbent solution, then poll step is termed as successful.

\item[Parameter Update:] If poll step is successful then the poll size parameter is either increased or kept constant but in case of failure it is reduced by some positive factor $\tau$ where $\tau >1$.
	A systematic reduction of poll size parameter by a positive factor during consecutive poll step failures and the choice of different positive basis at each iteration for direction generation ensures convergence to some local optimum.
\end{description}


\subsection{Generation of directions}
An important component in the poll step procedure of direct search is the generation of directions which, after normalization, are dense on the unit sphere.
These directions are generated from a positive spanning set using approaches based on $n+1$ equiangular directions \cite{alberto_nogueira_ea_pattern}  in $\mathbb{R}^n$ and  $2n$ orthogonal directions \cite{abramson_audet_ea_orthomads} in $\mathbb{R}^n$.
The dense set however, is generated by rotating these directions along a new direction at each iteration.
This new direction can be generated using Halton sequence \cite{halton1960efficiency}, Sobol sequence \cite{sobol1976uniformly, fasano2014linesearch}, simplex division approach \cite{edelsbrunner2000edgewise} or random operation \cite{audet_dennis_jr_mesh}.
Rotation along these directions can be carried out using Householder transformation \cite{golub2012matrix}.

We now give a brief overview of quadratic models, simplex gradient and scaled conjugate gradient which will be needed for the design of our approach.
\subsection{Quadratic Models}\label{quadModels}
Consider a sample set $Y$ of $p+1$ interpolation points i.
e.
$Y=\{y^0,y^1,\ldots,y^p\}$ and a polynomial basis $\phi$ of degree less than or equal to $2$ in $\mathbb{R}^n$.
So $\phi$ can be expressed as natural basis of monomials i.e.
\begin{align}
	\phi &= \{\phi_0(x),\phi_1(x),\ldots,\phi_q(x)\}\notag\\
	&= \{1,x_1,x_2,\ldots, x_n, \frac{x^2_1}{2},\frac{x^2_2}{2},\ldots,\frac{x^2_n}{2}, x_1x_2, x_1x_3,\ldots,x_{n-1}x_n\},
\end{align}
where $q+1=\frac{(n+1)(n+2)}{2}$ is the number of elements in the polynomial basis.
A quadratic model $m(y)=\alpha^T\phi(y)$ built over the set $Y$ should satisfy the interpolation condition:
\begin{equation}\label{quadInterpolateEqn}
	M(\phi,Y)\alpha=f(Y),
\end{equation}
where $f(Y)=(f(y^0),f(y^1),\ldots,f(y^p))$, set of function values of points in $Y$ and 
\begin{equation*}
	M(\phi,Y)=
	\begin{bmatrix}
		\phi_0(y^0) & \phi_1(y^0) & \ldots & \phi_q(y^0)\\
		\phi_0(y^1) & \phi_1(y^1) & \ldots & \phi_q(y^1)\\
		\vdots      & \vdots      & \ddots &  \vdots    \\
		\phi_0(y^p) & \phi_1(y^p) & \ldots & \phi_q(y^p)
	\end{bmatrix}.
\end{equation*}

Based on the number of sample points $p+1$ and $q$, there are three possible scenarios for solving the above system of linear equations \ref{quadInterpolateEqn}:
\begin{enumerate}
	\item When $p>q$, we have overdetermined system which can be solved in least squares sense i.e.
		\begin{equation}\label{quadLSeqn}
			\underset{\alpha \in \mathbb{R}^{q+1}}{\text{min}} ||M(\phi,Y)\alpha-f(Y)||^2,
		\end{equation}
	\item When $p=q$, we have determined system which can be solved directly.
	\item When $p<q$, we have underdetermined system.
\end{enumerate}

In practice, the first two scenarios rarely occur for large dimension problems because of limited function evaluation budget. 
Accordingly, there exist infinite possible solutions to the underdetermined system.
A possible solution can be derived \cite{connBook,custodio2010incorporating} by using Minimum Frobenius Norm (MFN).
For underdetermined case, it was shown \cite{connBook} that the error between $f$ and $m$, and between $\nabla f$ and $\nabla m$ is upper bounded by terms dependent on the norm of the Hessian of model $m$.
Hence, a model $m$ with least Hessian norm, is desirable.
The elements of Hessian, which are essentially quadratic terms of $\alpha$, can be minimized by building models based on MFN. 
So, bifurcating $\alpha$ into linear terms $\alpha_L \in \mathbb{R}^{n+1}$ and quadratic terms $\alpha_Q \in \mathbb{R}^{n_Q}$ where $n_Q=\frac{n(n+1)}{2}$, we have $m(y)=\alpha_L^T \phi_L+\alpha_Q ^T \phi_Q$ where $\phi_L=\{1,x_1,\ldots,x_n\}$ and $\phi_Q=\{\frac{x^2_1}{2},\frac{x^2_2}{2},\ldots,\frac{x^2_n}{2}, x_1x_2, x_1x_3,\ldots x_{n-1}x_n\}$.
The solution $\alpha$ is then obtained by solving the optimization problem:
\begin{align}\label{MFN}
	&\underset{\alpha_Q\in \mathbb{R}^{n_Q}}{\text{min }} \frac{1}{2}||\alpha_Q||^2\notag\\
	& \text{s.t. } M(\phi_L,Y)\alpha_L+M(\phi_Q,Y)\alpha_Q=f(Y).
\end{align}
The quadratic minimization problem reduces to solving a linear system
\begin{equation}
	F(\phi,Y)\begin{bmatrix}
		\mu \\
		\alpha_L
	\end{bmatrix}
	=
	\begin{bmatrix}
		f(Y)\\
		0
	\end{bmatrix},
\end{equation}
where
\begin{equation*}
	F(\phi,Y)=\begin{bmatrix}
		M(\phi_Q,Y)M(\phi_Q,Y)^T & M(\phi_L,Y) \\
		M(\phi_L,Y)^T & 0
	\end{bmatrix}
\end{equation*}
The linear term $\alpha_L$ is obtained directly by solving above linear system while quadratic term $\alpha_Q$ is obtained by computing $\alpha_Q=M(\phi_Q,Y)^T\mu$.

\subsection{Simplex Gradient}
	A sample set $Y=\{y^1, \ldots y^q\}$ of $q$ points in $\mathbb{R}^n$ such that they all lie within a ball $B(y^0,\Delta)$ of radius $\Delta \in \mathbb{R}_+ $ with center $y^0 \in \mathbb{R}^n$, is said to be poised \cite{custodio2007using} if rank$(Y)=\text{min}\{n,q\}$, or $Y$ has full column rank.
	The simplex gradient \cite{custodio2007using,Bortz98thesimplex} $\nabla_sf(y^0)$ of $f$ at $y^0$ is the solution to the linear system
	\begin{equation}\label{simplexGradEqn}
		Y^T\nabla_sf=b
	\end{equation}
where $b=[f(y^1)-f(y^0),\ldots,f(y^q)-f(y^0)]^T$.
If $q>n$, then the simplex gradient is obtained by solving a least square problem to above system.
The error  between a simplex gradient and actual gradient is upper bounded by $\Delta$ \cite{custodio2007using} such that
	\begin{equation}
		\|\nabla f(y^0)- \nabla_s f(y^0)\| \leq \eta_{eg} \Delta,
	\end{equation}
	where, $\eta_{eg}$ is some constant dependent upon $Y$ and $\Delta$.\\
	Simplex gradients are, in a broad sense, linear models for interpolation.

\subsection{Scaled Conjugate Gradient}
Conjugate gradient algorithm \cite{hager_zhang_survey} is a prominent approach for solving nonlinear optimization problems. 
For our work, we use a particular variant \cite{andrei2008unconstrained} of the conjugate gradient approach.
For minimizing a nonlinear function $f(x)$ with a known initial point $x_0$, this algorithm generates a sequence of points $x_i$ such that:
	\begin{align}\label{scg}
		x_{i+1}&=x_i+\alpha_i d_i \notag\\
		d_{i+1}&=-\theta_{i+1}g_{i+1}+\beta_i s_i \notag\\
		d_0 &=-g_0
	\end{align}
where $\alpha_i$ is a positive step size obtained by line search along $d_i \in \mathbb{R}^n$ direction and $g_i=\nabla f(x_i)$ is the gradient of $f$ at $x_i$, and $s_i=x_{i+1}-x_{i}$.
$\theta_{i+1}$ is a matrix or scalar parameter and $\beta_i$ is a scalar.
A value for $\beta_i $\cite{birgin2001spectral} in terms of $\theta_{i+1}, \, s_i$ and $y_i=g_{i+1}-g_i$ is:
\begin{equation}\label{betaCG}
	\beta_i=\frac{(\theta_{i+1}y_i-s_i)^Tg_{i+1}}{y_i^Ts_i}
\end{equation}

\section{Theory}\label{Theory}
We now state our method for solving box constrained derivative free optimization problems, which is composed of novel poll step and search step methods.
\subsection{Poll Step}
Our poll step comprises of following steps:
\begin{description}
	\item [Handling Box Constraints:] We use extreme barrier approach \cite{audet_dennis_jr_mesh} to handle the box constraints.
		In this approach, if an infeasible point is encountered then its function value is set to infinity.
		However, for practical purposes, we assign a large value to $f(x)$ for e.g. $1.79\times10^{308}$.
	\item [Equiangular Directions:] We generate a set $Y_0$ of $n+1$ equiangular directions about the origin.
		These directions are later translated along to current iterate $x_k$ to create a new set $Y_k$.

	\item [Sufficient Decrease and Parameter Update:] New trial points are generated along the equiangular directions, translated to incumbent solution $x_k$ and scaled appropriately to direct search radius $r_k$, where $r_k$ is upper bounded by poll size parameter $\Delta^p_k$ i.e. $r_k \leq \Delta^p_k$.
		Poll step is considered successful if there is a sufficient decrease in function values at these points compared to $f(x^k)$ i.e.
		\begin{equation}\label{suffdecrease}
			f(x_{k+1}) < f(x_k)-\gamma(r_k)
		\end{equation}
		where function $\gamma:\mathbb{R}_+ \rightarrow \mathbb{R}_+$ is defined as $\gamma(r)=\rho r^2$ and $r_k$ is the direct search radius at iteration $k$.
		$\rho$ is a positive scalar parameter such that $\rho <1$.

		In the event of failure of poll step, $\Delta^p_k$ is scaled down to $\frac{\Delta^p_k}{\tau_p}$ where $\tau_p>1$.
		Consequently $r_k$ is also updated as $r_{k+1}=\frac{r_k}{\tau_p}$.

	\item [Rotation of directions:] If the poll step is unsuccessful, a new direction $u_k$ is generated using Halton sequence.
		A new set $Y_{k+1}$ of $n+1$ equiangular direction vectors are generated by rotating the directions in $Y_0$ using Householder transformation to contain the direction $u_k$.
		The rotation of minimal positive basis using Householder transformations was first suggested in \citet{alberto_nogueira_ea_pattern}.
		This is accomplished by a matrix vector multiplication i.e. $d_{k+1}=H_{k}d_k$ where  $d_k \in Y_k$, $d_{k+1} \in Y_{k+1}$ and matrix $H_k \in \mathbb{R}^{n\times n}$ is given by:
		\begin{equation}\label{householder}
	H_k=I-2\frac{(d_k-u_k)(d_k-u_k)^T}{||d_k-u_k||^2}.
\end{equation}
		These directions are then scaled to length $r_{k+1}$ and new trial points are generated along them.
        Subsequently function values are evaluated at them and checked for sufficient decrease condition mentioned in Eq.(\ref{suffdecrease}).
\end{description}

\begin{algorithm}[t]
	\caption{Poll Step}\label{pollStep}
	\begin{algorithmic}[1]
		\STATE \textbf{INPUT}
		\STATE The incumbent solution $x_k$, direct search radius $r_k$, $0< \rho <1$, $\tau_l>1$, poll size parameter $\Delta^p_k$ and threshold $\epsilon >0$.
		\STATE A generated set $Y_0$ of $n+1$ equiangular directions of unit length about origin.
		\WHILE {$r_k > \epsilon$}
		\STATE Translate directions in $Y_0$ with $x_k$ to create $Y_k$ i.e. set $y_i \in Y_k$ as $y_i=y_i+x_k$ for $i=1,\ldots n+1$.
		\STATE Compute the set of $n+1$ points $A_k=\{x_1,\ldots,x_{n+1}\}$ along these directions, and scale them such that they lie on the ball $\mathcal{B}(x_k,r_k)$ where is $r_k \leq \Delta^p_k$.
		\STATE Evaluate function values at these points. Set $f(x)=\infty$ if $x\notin \Omega$.
		\IF {Sufficient decrease condition \ref{suffdecrease} is satisfied}
			\STATE \textbf{RETURN:} Best point $\underset{x_i \in A}{\text{argmin}}\{f(x_i)\}$, poll step radius $r_k$ and $\Delta^p_k$.
		\ELSE
			\STATE Using Halton sequence generate a new direction $u_k$.
			\STATE Using Householder transformation equation \ref{householder} rotate the equiangular directions of set $Y_0$ about $u_k$.
			\STATE Let $Y_k$ be the set of these directions.
			\STATE $k \gets k+1$
			\STATE Update $\Delta^p_{k+1}=\frac{\Delta^p_k}{\tau_l}$ and $r_{k+1}=\frac{r_k}{\tau_l}$ 
		\ENDIF
		\ENDWHILE
		\IF {$r_k < \epsilon$}
		\STATE Stationary point achieved.
		\STATE Terminate.
		\ENDIF
	\end{algorithmic}
\end{algorithm}
Algorithm \ref{pollStep} terminates whenever the sufficient decrease condition is satisfied returning the best point and direct search radius.

\subsection{Search Step}
Our approach consists of following important steps:
\begin{description}

	\item [Quadratic Model Step:]
		We build a quadratic model around the trial points and the incumbent solution $x_k$, generated in poll step. 
		This includes all points generated during consecutive failures of poll step about $x_k$.
		Least squares or MFN models are chosen based on the number of points available. 
		If Hessian of the quadratic model is positive definite, we compute its unique minima $y_k$ by solving a linear equation.
		If the new minima is infeasible i.e. outside the bound constraints, we try to obtain a new solution by carrying out line search along $y_k-x_k$ such that its solution lies within the bound constraints.
		If Hessian of the model is not positive definite, no attempt is made to compute its minima.

	\item [Simplex Gradient:]
		We compute the simplex gradient $\nabla_sf_k$ around $x_k$. 
		We include all available feasible points within a ball of radius $\Delta^p_k$ and center $x_k$.
		If number of points are greater than $n+1$ then it is computed in the regression sense i.e. by solving a least square problem.
		We evaluate function value at the point say $x_g$, in the direction $-\nabla_sf_k$, at a distance $\Delta^p_k$ from $x_k$.

	\item [Vicinity Search:]
		We sort all the points in the last poll step according to their function values in ascending order.
		We choose first $l$ points from the list.
		Now we take the average of these points with the point having best function value.
		We generate $l$ points along directions from $x_k$ to these averages such that their distance from $x_k$ is $r_k$ and evaluate function values at them.

	\item [Scaled Conjugate Gradient Step:]
		Let $x_b$ be the point with best function value at iteration $k$.
		Then compute the positive semidefinite matrix $\theta=-\frac{(x_b-x_k)(x_b-x_k)^T}{(x_b-x_k)\nabla_sf_k}$.
		We now compute $\beta$ from equation \ref{betaCG} and corresponding new search direction for conjugate gradient.
		We then carry out a line search along the new direction inside the feasible region $\Omega$ to obtain a point with better function value.

\end{description}

\subsection{Final Algorithm}
As the algorithm is constituted of Scaled Conjugate Gradient (SCG) and direct search (DS) for solving box constrained (BC) DFO, we name our algorithm as BCSCG-DS. 
We now summarize our approach BCSCG-DS in Algorithm \ref{cscgmadsAlgo}.

\begin{algorithm}
	\caption{BCSCG-DS algorithm}\label{cscgmadsAlgo}
	\begin{algorithmic}[1]
		\STATE \textbf{Initialization}:
		\STATE Set $x^0 \in \Omega$ as the center or starting point.
		\STATE $0< \rho <1$, $\tau_p>1$, $0<\epsilon_2<1$, $\tau_u>1$ and threshold $\epsilon >0$.

		\STATE Initial poll size parameter $\Delta_0^p >0$ and direct search radius $0<r_0 \leq \Delta_0^p$.
		\STATE Set $l$, the number of points to be considered for vicinity search step.
		\STATE  Generate a set $Y_0$ of $n+1$ equiangular directions of unit length about origin.
		\STATE \textbf{Iteration $k$}
		\WHILE {Termination Criteria of poll step is not met}
		\STATE Do \textbf{Poll Step} using algorithm \ref{pollStep}.

		\STATE Begin \textbf{Search Step}
		\STATE \textbf{Quadratic Model}
		\IF {$x_i \in \Omega$ for all $x_i \in A_k$}
			\STATE Fit a quadratic model $m$ over all feasible points generated during the current poll step.
			\STATE If hessian of $m$ is positive definite then find its minima $y_k$.
			\STATE If $y_k\notin \Omega$ do line search along $y_k-x_k$ to find a good feasible solution 
			\STATE Update the best point $x_b$.
		\ENDIF

		\STATE \textbf{Vicinity Search}
		\IF {$x_i \in \Omega$ for all $x_i \in A_k$}
			\STATE Collect all previous evaluated feasible points within the ball $\mathcal{B}(x_k,r_k(1+\epsilon_2))$.
			\STATE Compute simplex gradient ``$\nabla_sf$'' at $x_k$ from these points using linear system \ref{simplexGradEqn} 
			\STATE Compute point $x_g$ along $-\nabla_sf$ such that $\|x_g-x_k\|_2=r_k$ and add it to set $A_k$
		\ELSE
			\STATE Set $x_g=x_b$.
		\ENDIF
		\STATE Sort the points in $A_k$ in ascending order of their function values.
		\STATE Select first $l$ points. Let their set be $V$.
		\STATE Compute the average of best point $x_b$ with points in $V$ and store them in $\bar{V}$.
		\STATE Evaluate function values for points in $\bar{V}$ and update the best point $x_b$.

		\STATE \textbf{Scaled Conjugate Gradient}
		\STATE Compute matrix $\theta_k=-\frac{(x_b-x_k)(x_b-x_k)^T}{(x_b-x_k)\nabla_sf_k}$.
		\STATE Evaluate a new direction using equations \eqref{scg} and \eqref{betaCG}.
		\IF {The new direction is descent}
			\STATE Do line search along it within the box constraints.
		\ELSE
			\STATE Do line search along $x_b-x_k$ within the box constraints.
		\ENDIF
		
		\IF {Steplength $>\Delta^p_k$}
			\STATE Update $\Delta^p_{k+1}=$Steplength.
		\ENDIF
		\IF {Steplength $>2\Delta^p_k$}
			\STATE Update $\Delta^p_{k+1}=\tau_u\Delta^p_k$.
		\ENDIF
		\STATE Update the best point $x_b$ and set center to best point i.e. $x_k \gets x_b$
		\STATE $k \gets k+1$ and goto poll step 
	\ENDWHILE
	\end{algorithmic}
\end{algorithm}
For practical purposes, the termination criteria for Algorithm \ref{pollStep} is generally set to the budget on the number of function evaluations.

We will now give some results required to show the convergence of the proposed algorithm to some Clarke-Jahn stationary point.

\begin{lemma}
	\label{finiteSearchStep}
	Number of function evaluations done during the search step of Algorithm \ref{cscgmadsAlgo} is finite.
\end{lemma}
\begin{proof}
	Search step is activated when poll step is successful. So there exists $\alpha_k \in \mathbb{R}_+$ and $d_k \in \mathbb{R}^n$ such that  $f(x_k+\alpha_k d_k) < f(x_k)-\rho \alpha_k^2$ where $x_k$ is the current incumbent solution.
	The next step requires the use of quadratic model over the points generated during the latest poll step, and it requires at most one new function evaluation.
	As per the construction of algorithm, vicinity search too requires finite number of function evaluations.
	Next step involves the computation of simplex gradient and one function evaluation along the direction opposite to it.
	In case negative of gradient is a descent direction, scaled conjugate gradient method is evoked which uses a line search.
	In case it is not a descent direction, algorithm follows a greedy approach and does line search along some descent direction, which is necessarily present since the last poll step was successful.
	Line search is done using Brent algorithm, which by construction, is restricted to finite number of iterations or function evaluations.
	Thus all intermediate steps involved during the search step require finite number of function evaluations.
\end{proof}

We define $[x+\alpha d]_{[l,u]}=\text{max}\{l,\text{min}\{u,(x+\alpha d)\}\}$ where $x \in \mathbb{R}^n$, $d\in \mathbb{R}^n$, $\alpha \in \mathbb{R}_+$ and $l$ and $u$ represent lower and upper bound on $x$ i.e. $l_i \leq x_i \leq u_i \; \text{for }i=1,2,\ldots,n$.
Let $t\in \mathbb{N}\cup \{0\}$ and $\tau_\beta \in \mathbb{R}$ such that $\tau_\beta \geq 1$.
Also, let $\alpha$ be the steplength such that $x+\alpha d \in \Omega$ and $f(x+\alpha d) \leq f(x)-\rho \alpha^2$.
Let us define projection parameter $\beta \in \mathbb{R}$ as:
\begin{equation}
	\beta= \tau_\beta^{t_\beta} \alpha \text{ where } t_\beta=\{\underset{t \in \mathbb{N}}{\text{arg max }}f([x+\tau_\beta^t\alpha d]_{[l,u]}) \; : \; f([x+\tau_\beta^t\alpha d]_{[l,u]}) \leq f(x)-\rho (\tau_\beta^t \alpha)^2\},\\
\end{equation}
and extended projection parameter $\eta \in \mathbb{R}$ as:
\begin{equation}
	\eta= \tau_\beta^{t_\eta} \alpha \text{ where } t_\eta=\{\underset{t \in \mathbb{N}}{\text{arg min }}f([x+\tau_\beta^t\alpha d]_{[l,u]}) \; : \; f([x+\tau_\beta^t\alpha d]_{[l,u]}) > f(x)-\rho (\tau_\beta^t \alpha)^2\}.\\
\end{equation}

Clearly, $\beta \geq 0$ and $\eta \geq 0$ since $\alpha > 0$, and are related as
\begin{equation}
	\label{betaEtaRelation}
	\eta=\tau_\beta \beta.
\end{equation}

\begin{lemma}
	\label{betaEtaFiniteness}
	At any iteration $k$ of Algorithm \ref{cscgmadsAlgo}, projection parameter $\beta_k$ and extended projection parameter $\eta_k$ are finite.
\end{lemma}
\begin{proof}
	Let us assume that $\beta_k$ and thus $t_\beta$ to be not finite.
	Since $\Omega$ is compact, from Weierstrass theorem, there exists $\bar{x}\in \Omega$ such that $\bar{x}=\underset{x \in \Omega}{\text{arg min }}f(x)$ and $f(\bar{x}) > -\infty$.
	Since $\rho > 0$ and $\tau_\beta \geq 1$, there exists $\bar{t}\in \mathbb{N}$ such that 
	\begin{equation*}
		f(x_k)-\rho (\tau_\beta^{\bar{t}}\alpha_k)^2 < f(\bar{x}),
	\end{equation*}
	or
	\begin{equation*}
		f([x_k+\tau_\beta^{\bar{t}}\alpha_k d_k]_{[l,u]})-\rho (\tau_\beta^{\bar{t}}\alpha_k)^2 < f(\bar{x}),
	\end{equation*}
	which is a contradiction, since $\bar{x}$ is the minimum.
	So, $t_{\beta}$ and hence $\beta_k$ are finite.
	Similarly, from equation \ref{betaEtaRelation} we have $\eta_k$ to be finite.
\end{proof}

\begin{lemma}
	\label{PollStepRadiusConverge}
	The sequence of poll step parameter $\{\Delta_k^p\}$, generated by Algorithm \ref{cscgmadsAlgo}, in the limit, converges to $0$ i.e. $\underset{k \rightarrow 0}{\text{lim }}\Delta_k^p=0$. Also, 
	\begin{align}
		\underset{k \rightarrow 0}{\text{lim }}\beta_k &=0 \\
		\underset{k \rightarrow 0}{\text{lim }}\eta_k &=0.
	\end{align}
	where $\beta_k$ and $\eta_k$ are projection parameter and extended projection parameter at iteration $k$.
\end{lemma}
\begin{proof}
	Let $P_1$ be the sequence of iterations of Algorithm \ref{cscgmadsAlgo} during which poll step parameter is increased i.e. $\Delta^p_{k+1}=\tau_u \Delta^p_{k}$ where $\tau_u > 1$.
	Similarly, let $P_2$ be the sequence of iterations during which poll step parameter is decreased i.e.  $\Delta^p_{k+1}=\frac{\Delta^p_{k}}{\tau_l}$ where $\tau_l >1$.
	Now, by construction of algorithm, $P_1 \cup P_2$ cannot be a finite sequence.
	From lemma \ref{finiteSearchStep}, all intermediate directions and step lengths during the search step can be replaced by a single direction and step length.
	Following two situations arise:
	\begin{description}
		\item[Case-1: $P_1$ is infinite:] From sufficient decrease condition of Algorithm \ref{cscgmadsAlgo},
			\begin{equation*}
				f(x_{k+1}) = f(x_k+\alpha_k d_k) \leq f(x_k)-\rho \alpha^2_k.
			\end{equation*}
			Since $\Omega$ is compact and $f$ is continuous, from Weierstrass theorem, we have
			\begin{equation*}
				\underset{k\rightarrow \infty}{\text{lim }}f(x_k)=\bar{f}. 
			\end{equation*}
			Thus, $\underset{k \rightarrow \infty}{\text{lim }}\rho \alpha^2_k=0$ or $\underset{k \rightarrow \infty,\,k\in P_1}{\text{lim }} \alpha_k=0$.
			Since $\Delta_k^p$ is always upperbounded by steplength $\alpha_k$, we have $\underset{k \rightarrow \infty,\,k\in P_1}{\text{lim }}\Delta_k^p=0$.
			Also, since $\beta_k=\tau_\beta^t \alpha_k$ for some $t\in \mathbb{N}$ and $\eta_k=\tau_\beta^{t+1} \alpha_k$ (from equation \ref{betaEtaRelation}),
			\begin{equation*}
				\underset{k \rightarrow \infty,\,k\in P_1}{\text{lim }}\beta_k=0 \qquad\text{ and }\qquad \underset{k \rightarrow \infty,\,k\in P_1}{\text{lim }}\eta_k=0.
			\end{equation*}

		\item[Case-2: $P_1$ is finite:] Clearly $P_2$ is a infinite sequence.
			Let $|P_1|=s$ where $s \in \mathbb{N}\cup\{0\}$. So, for $k > s$, 
			\begin{equation*}
				\alpha_{k+1}=\frac{\alpha_s}{\tau_l^{(k-s)}}.
			\end{equation*}
			Since $\tau_l > 1$, we have $\underset{k \rightarrow \infty,\,k\in P_2}{\text{lim }}\alpha_k=0$ or $\underset{k \rightarrow \infty,\,k\in P_2}{\text{lim }}\Delta_k^p=0$.
			Also, $\underset{k \rightarrow \infty,\,k\in P_2}{\text{lim }}\beta_k=0$ and $\underset{k \rightarrow \infty,\,k\in P_2}{\text{lim }}\eta_k=0$.

	\end{description}
\end{proof}

We now define a lemma from \cite{fasano2014linesearch} [see lemma-2.6]. 
\begin{lemma}
	\label{limitPointDirectionSL}
	Let $\{x_k\}_P$, $\{d_k\}_P$ and $\{\alpha_k\}_P$ be some sequences with $P \subseteq \mathbb{N}$ being their subset of indices such that $x_k \in \Omega$, $d_k \in D(x_k)$, $\alpha_k \in \mathbb{R}_+$ and $x_{k+1}=[x_k+\alpha_k d_k]_{[l,u]}$.
	Further, let us assume that there exists a subset $\bar{P} \subseteq P$ such that
	\begin{align}
		\underset{k\rightarrow \infty,k\in \bar{P}}{\text{lim }} x_k &=\bar{x},\\
		\underset{k\rightarrow \infty,k\in \bar{P}}{\text{lim }} d_k &=\bar{d},\\
		\underset{k\rightarrow \infty,k\in \bar{P}}{\text{lim }} \alpha_k &=0
	\end{align}
	where $\bar{x} \in \Omega$ and $\bar{d} \in D(\bar{x}),\bar{d} \neq 0$. Then,
	\begin{enumerate}
		\item there exists $m \in \mathbb{N}$ such that for all $k > m$,
			\begin{equation*}
				[x_k+\alpha_k d_k]_{[l,u]} \neq x_k;
			\end{equation*}
		\item for $w_k=\frac{[x_k+\alpha_k d_k]_{[l,u]}-x_k}{\alpha_k}$, we have
			\begin{equation*}
				\underset{k\rightarrow \infty,k\in \bar{P}}{\text{lim }} w_k = \bar{d},
			\end{equation*}
	\end{enumerate}
\end{lemma}

We now present the main result of this work. 
Note that the approach presented here is similar to the one presented in \cite{fasano2014linesearch} [see Proposition 2.7].
\begin{proposition}
	Let $\{x_k\}$ be the sequence generated by Algorithm \ref{cscgmadsAlgo}.
	Let us assume that there exists $\bar{x}\in \mathbb{R}^n$ such that $\underset{k \rightarrow \infty,\, k \in P}{\text{lim }}x_k=\bar{x}$ where $P$ is some set of indices.
	Point $\bar{x}$ is said to be Clarke-Jahn stationary point if $\{d_k\}_P$, the sequence of normalized directions generated by the algorithm, is dense in the unit sphere.
\end{proposition}
\begin{proof}
	The Clarke-Jahn stationarity condition at $\bar{x}$ states that for every normalized direction $d\in D(\bar{x})$, $f^\circ(\bar{x};d) \geq 0$, i.e.
	\begin{equation*}
		\underset{\underset{t \downarrow 0,\,y+td \in \Omega}{y \rightarrow \bar{x},\, y \in \Omega}}{\text{lim sup }} \frac{f(y+td)-f(y)}{t} \geq 0.
	\end{equation*}
	We now prove the result using contradiction, by assuming the existence of some direction $\bar{d} \in D(\bar{x}) \cap \mathcal{B}(0,1)$ such that
	\begin{equation}
		\label{ClarkeJahnIneq}
		\underset{\underset{t \downarrow 0,\,y+t\bar{d} \in \Omega}{y \rightarrow \bar{x},\, y \in \Omega}}{\text{lim sup }} \frac{f(y+t\bar{d})-f(y)}{t} < 0.
	\end{equation}
	Let $\eta_k$ be the extended projection parameter at iteration $k$.
	Let $w_k=\frac{[x_k+\alpha_k d_k]_{[l,u]}-x_k}{\alpha_k}$.
	Now by lemma \ref{PollStepRadiusConverge}, we have $\underset{k\rightarrow \infty}{\text{lim }}\eta_k=0$.
	Since $\{d_k\}_P$ is dense in unit sphere, there exists a subset of indices $\bar{P} \subseteq P$ such that  $\underset{k\rightarrow \infty,\, k \in \bar{P}}{\text{lim }}d_k=\bar{d}$.
	Additionally, $\underset{k\rightarrow \infty,\, k \in \bar{P}}{\text{lim }}x_k=\bar{x}$ and $\underset{k\rightarrow \infty,\, k \in \bar{P}}{\text{lim }}\eta_k=0$.
	We can see that all the assumptions of lemma \ref{limitPointDirectionSL} are satisfied.
	So there exists $m \in \mathbb{N}$ such that for all $k > m$ with $k \in \bar{P}$, $w\neq 0$ and $\underset{k\rightarrow \infty,k\in \bar{P}}{\text{lim }} w_k = \bar{d}$. Thus, $f(x_k+\eta_k w_k) > f(x_k)-\rho \eta_k^2$ and since $\eta_k > 0$,
	\begin{equation}
		\label{rhoEta}
		\frac{f(x_k+\eta_k w_k)-f(x_k)}{\eta_k} > -\rho \eta_k,
	\end{equation}
for $k > m$.

Let $L$ be the Lipschitz constant of $f$. Now, 
	\begin{align*}
		& \underset{\underset{t \downarrow 0,\,x_k+t\bar{d} \in \Omega}{x_k \rightarrow \bar{x},\, x_k \in \Omega}}{\text{lim sup }} \frac{f(x_k+t\bar{d})-f(x_k)}{t} \\
		& \geq \underset{k\rightarrow \infty,\,k \in \bar{P}}{\text{lim sup }} \frac{f(x_k+\eta_k \bar{d})-f(x_k)}{\eta_k}\\
		& = \underset{k\rightarrow \infty,\,k\in \bar{P}}{\text{lim sup }}\frac{f(x_k+\eta_k \bar{d})+f(x_k+\eta_k w_k)-f(x_k+\eta_k w_k)-f(x_k)}{\eta_k}\\
		& \geq \underset{k\rightarrow \infty,\,k \in \bar{P}}{\text{lim sup }} \frac{f(x_k+\eta_k w_k)-f(x_k)}{\eta_k}-\underset{k\rightarrow \infty,\,k \in \bar{P}}{\text{lim }} \|\bar{d}-w_k\|\\
		& > -\rho \underset{k\rightarrow \infty,\,k \in \bar{P}}{\text{lim }} \eta_k-\underset{k\rightarrow \infty,\,k \in \bar{P}}{\text{lim }} \|\bar{d}-w_k\|,
	\end{align*}
	where the last inequality follows from \ref{rhoEta}.
	Now, from lemma \ref{limitPointDirectionSL}, we have $\underset{k\rightarrow \infty,\,k \in \bar{P}}{\text{lim }} \|\bar{d}-w_k\|=0$.
	Since $\underset{k\rightarrow \infty,\,k \in \bar{P}}{\text{lim }} \eta_k =0$, we have
	\begin{equation*}
		\underset{\underset{t \downarrow 0,\,x_k+t\bar{d} \in \Omega}{x_k \rightarrow \bar{x},\, x_k \in \Omega}}{\text{lim sup }} \frac{f(x_k+t\bar{d})-f(x_k)}{t} \geq 0,
	\end{equation*}
	which is a contradiction to our assumption \ref{ClarkeJahnIneq}.
	Thus, $\bar{x}$ is a Clarke-Jahn stationary point for $f$.

\end{proof}

\section{Numerical Results}\label{Results}

We performed a comparative testing of our method BCSCG-DS with the state of art solver Bobyqa \cite{powell2009bobyqa}, which is fast and has the ability to handle high dimensional problems.
We compared the algorithms on the basis of number of function evaluations required to reach some local optimum within the given budget.

For Bobyqa, we set RHOBEG$=0.2$, RHOEND$=10^{-6}$, WORKSPACE$=10^8$ and other parameters to their respective defaults.
The computational budget option MAXFUNC was set to $40(n+1)$ function evaluations.

We implemented our algorithm BCSCG-DS in C using gcc compiler(version 4.9.2).
For performing the linear algebra computations, we used BLAS and LAPACK library \cite{laug}.
The tolerance i.e. minimum allowed poll size parameter was set to $10^{-6}$ and computation budget set to $40(n+1)$.
We set the initial poll size $\Delta_0^p$ and direct search radius $r_0$ to $0.1\,\underset{i=1,\ldots,n}{\text{min}}\{u_i-l_i\}$ and scalar parameter $\rho$  to $0.25$. 
The parameter $\epsilon_2$ is set to $0.01$.
For the vicinity search, we chose $l=\lfloor0.1n\rfloor$ points.
We used Brent algorithm for performing line searches for which we set the maximum iterations to 20 and the tolerance to $10^{-5}$. 

Our test suite comprised of two classes of box constrained optimization problems: noisy smooth problems and noisy piecewise-smooth problems taken from the set of least square problems reported by \citet{lukvsan2018sparse}, who outline a diverse collection of problems.
We considered 55 least square problems to test the effectiveness of our proposed approach.
Table \ref{tab:testProblems} consists the list of least square functions considered for our computational experiments.
The lower and upper bound on each variable for entire test collection was set to $-50$ and $50$ respectively.
We now state the approach of generating piecewise-smooth problems from least square problems and adding noise to them as suggested in \cite{more_wild_benchmarking}.
A typical least square problem is of the form:
\begin{equation}
	f(x)=\sum_{i=1}^m f_i(x)^2,
\end{equation}
where $f_i:\mathbb{R}^n \rightarrow \mathbb{R} \quad \text{for } i=1,\ldots,m$ is a continuous function.
The piecewise-smooth problems were generated from least square problems by modifying each individual term in least square function with its absolute value i.e.
\begin{equation}
	f(x)=\sum_{i=1}^m | f_i(x)|.
\end{equation}
For adding the noise, we first define a cubic Chebyshev polynomial $U_3$ as
\begin{equation}
	U_3(\alpha)=\alpha(4\alpha^2-3),
\end{equation}
where $\alpha \in \mathbb{R}$.
Let $\psi:\mathbb{R}^n\rightarrow [-1,1]$ be a function defined as
\begin{equation}
	\psi(x)=U_3(\psi_0(x)),
\end{equation}
where
\begin{equation}
	\psi_0(x)=0.9\, \text{sin}\,(100\|x\|_1)\, \text{cos}\,(100\|x\|_\infty)+0.1\, \text{cos}\,(\|x\|_2)
\end{equation}
is a continuous and piecewise continuously differentiable with $2^n n!$ continuously differentiable regions.
We define the noisy problem $f(x)$ for the smooth case as
\begin{equation}
	f(x)=(1+\epsilon_f\psi(x))\sum_{i=1}^m f_i(x)^2
\end{equation}
and piecewise smooth case as
\begin{equation}
	f(x)=(1+\epsilon_f\psi(x))\sum_{i=1}^m |f_i(x)|
\end{equation}
where $\epsilon_f$ is the relative noise level.
For our experiments, we set $\epsilon_f=10^{-3}$ as suggested in \cite{more_wild_benchmarking}.

Since performance of algorithms is affected by the dimension of the problem, we studied each problem for large dimensions with  magnitude $200,250$ and $300$.
Hence, around $165$ smooth and $165$ piecewise-smooth problems were studied.
In order to remove the effects of starting point on the algorithm, we generated $10$ different feasible starting points for each problem and inference was drawn on the basis of their median.
Performance of a method is dependent on the final solution and the amount of function evaluations needed to achieve it.
We define normalized function evaluations as a multiple of $n+1$ function evaluations.
In our comparison, a method is considered to have better performance if it computes a lower function value within the computational budget of $40$ normalized function evaluations i.e. $40(n+1)$ function evaluations.

\subsection{Performance Profiles}
For comparative analysis, we computed the performance profiles \cite{more_wild_benchmarking,dolan2002benchmarking}, which are frequently used for quantitative assessment of derivative free optimization solvers.
Let $\mathcal{P}$ be the collection of test problems and $\mathcal{S}$ be the set of solvers or algorithms under consideration.
For a particular instance of a test problem, let $x_0$ be the starting point for all the solvers and $f(x_0)$ be its corresponding function value.
Let the optimal solution obtained by a solver $s \in \mathcal{S}$ within the given computational budget be $f^*_s$.
Then, the convergence condition for the solver $s$ is defined as:
\begin{equation}\label{solverConverge}
	f(x_0)-f^*_s \geq (1-\tau)(f(x_0)-f^*_L)
\end{equation}
where $0<\tau \leq 1$ is a user defined tolerance and $f^*_L=\text{min}\{f^*_s:s \in \mathcal{S}\}$ is the minimum function value obtained by any solver.
Let the performance measure for a solver $s\in \mathcal{S}$ for a particular problem instance $p \in \mathcal{P}$ be $w_{s,p}$ which can be the amount of time taken or the number of function evaluations performed by the solver to obtaining the result.
In this work, we consider $w_{s,p}$ to be number of function evaluations done.
Then, the performance of a particular solver $s$ with respect to the best solver, on a given problem instance, is given by \emph{performance ratio} which is defined as:
\begin{equation*}
	\nu_{s,p}=\frac{w_{s,p}}{\text{min}\{w_{s,p}:s\in \mathcal{S}\}}.
\end{equation*}
For a solver $s$ which is unable to satisfy the condition \ref{solverConverge}, we set $\nu_{s,p}=\infty$.
The performance profile for a solver $s\in \mathcal{S}$, is defined as the fraction of problems solved with respect to an upper bound $\alpha$ on $\nu$, i.e.,
\begin{equation*}
	\rho_s(\alpha)=\frac{1}{|\mathcal{P}|}\text{size}\{p \in \mathcal{P}:\nu_{s,p} \leq \alpha\}.
\end{equation*}
Performance profiles illustrate the relative performance of solvers, on the given problem set. 
Thus $\rho_s(1)$ represent the fraction of problems over which best performance was shown by solver $s$.
 We constructed the performance profiles for the smooth and piecewise-smooth problems with noise using tolerance $\tau=\{10^{-2},10^{-4}\}$ in Eq.(\ref{solverConverge}).

 Almost all problems utilized the complete budget of $40$ normalized function evaluations which is justifiable since the dimension of the test problems is quite high.
 Tables \ref{tab:perfSmoothNoisy} and \ref{tab:perfPiecewiseSmoothNoisy} show the percentage of problems on which each solver performed better than the other for the smooth and piecewise-smooth test problems with varying dimensions.
For both smooth and piecewise-smooth problems with noise, BCSCG-DS found better solution than Bobyqa in accordance with the test condition given in Eq.(\ref{solverConverge}) for dimensions $200,250$ and $300$.

\begin{table}[h]
 \begin{center}
	\begin{tabular}{|c|c|c|c|c|}
	\hline
	\multirow{2}{*}{Dimension} &
	\multicolumn{2}{c|}{$\tau=10^{-2}$} &
		\multicolumn{2}{c|}{$\tau=10^{-4}$}\\\cline{2-5}
	&Bobyqa&BCSCG-DS&Bobyqa&BCSCG-DS\\
	\hline
	200&25.45&90.91&25.45&89.09\\
	250&29.09&85.45&29.09&85.45\\
	300&30.91&85.45&30.91&83.64\\
	\hline
	\end{tabular}
 \end{center}
	\caption{Performance ratio (in terms of percentage) for noisy smooth problems} \label{tab:perfSmoothNoisy}
\end{table}

\begin{table}[h]
 \begin{center}
	\begin{tabular}{|c|c|c|c|c|}
	\hline
	\multirow{2}{*}{Dimension} &
	\multicolumn{2}{c|}{$\tau=10^{-2}$} &
		\multicolumn{2}{c|}{$\tau=10^{-4}$}\\\cline{2-5}
	&Bobyqa&BCSCG-DS&Bobyqa&BCSCG-DS\\
	\hline
	200&22.22&92.59&22.22&90.74\\
	250&27.77&85.18&27.77&83.33\\
	300&29.63&83.33&29.63&74.07\\
	\hline
	\end{tabular}
\end{center}
	\caption{Performance ratio (in terms of percentage) for noisy piecewise-smooth problems} \label{tab:perfPiecewiseSmoothNoisy}
\end{table}

The competitive performance of BCSCG-DS is clearly evident from the results of extensive computational experimentation.
In lot of cases, the gap between the solutions attained by BCSCG-DS and Bobyqa is significant. 
The above results clearly validate the applicability of BCSCG-DS, for solving smooth and piecewise-smooth derivative free optimization problems.
They also show its effectiveness and robustness on problems with varying dimensions for different tolerance $\tau$. 


\subsection{Progress Curves}
\begin{figure*}[t]
	\centering
	\includegraphics[scale=0.8]{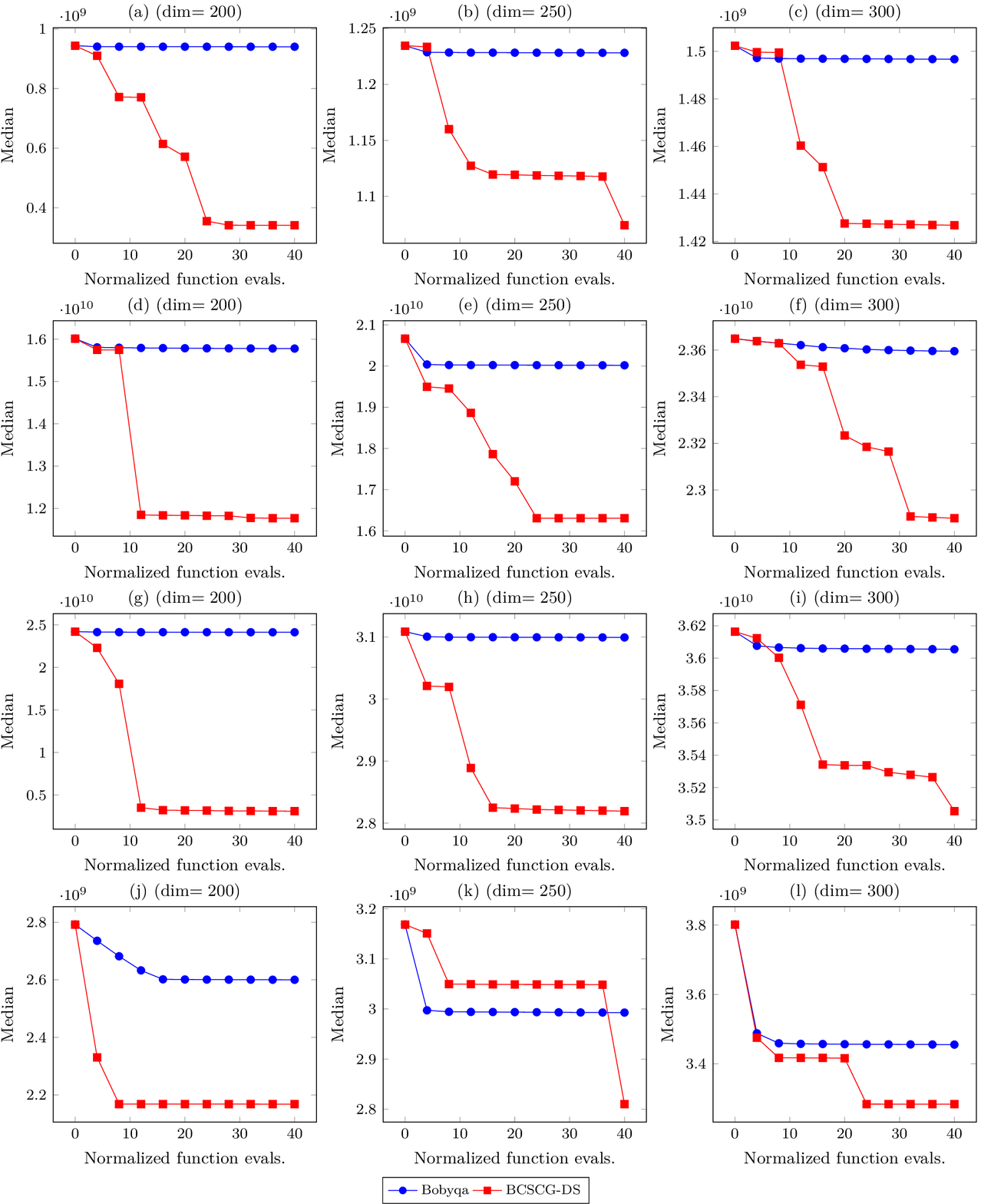}
	\caption{Progress curves for smooth problems with noise with varying dimensions for (a-c) Generalized Broyden Tridiagonal Problem (d-f) Chained and Modified problem HS53 (g-i) Attracting-Repelling Problem (j-l) Modified Countercurrent Reactors Problem-2}
	\label{fig:smoothnoisy1}
\end{figure*}

\begin{figure*}[t]
	\centering
	\includegraphics[scale=0.8]{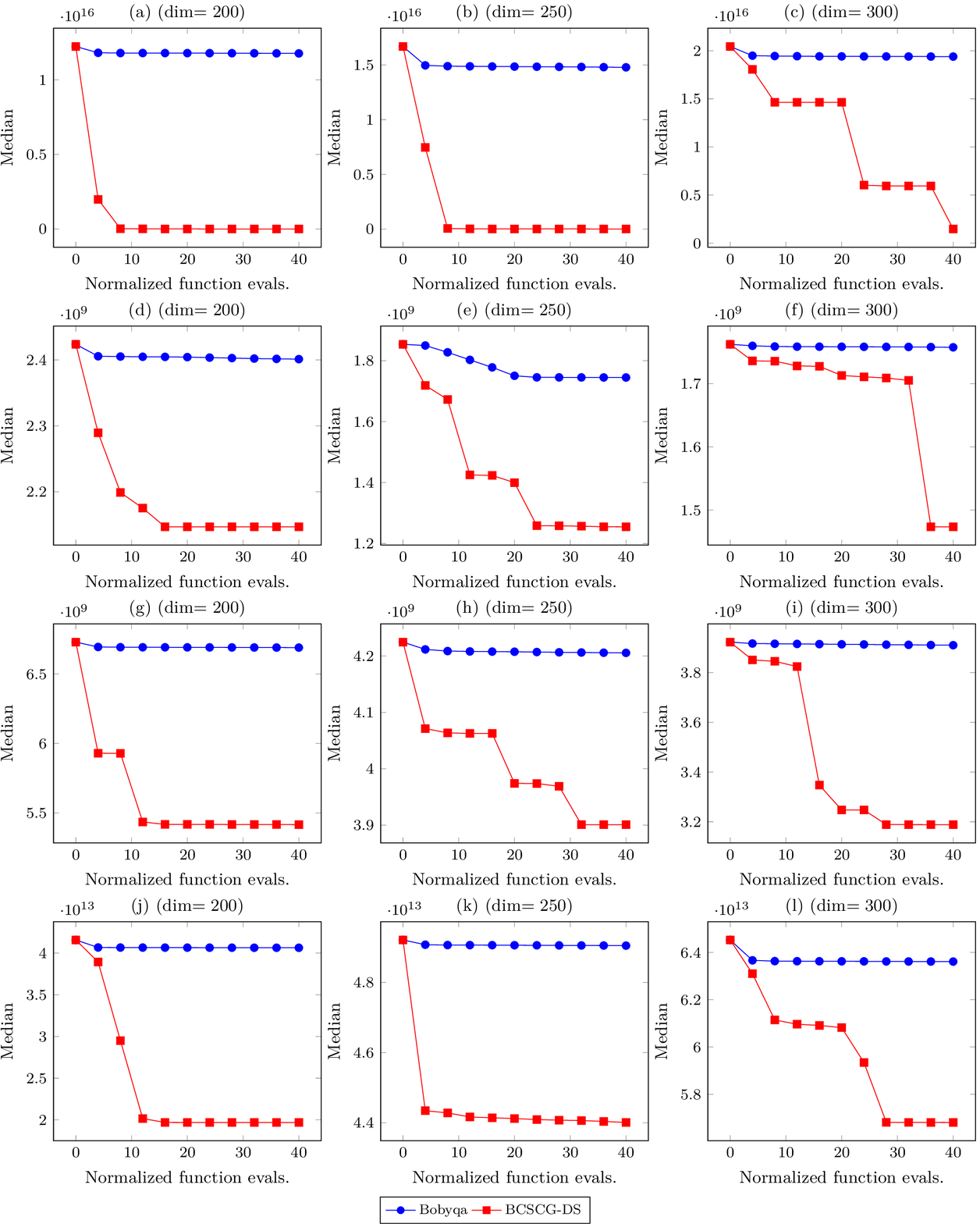}
	\caption{Progress curves for smooth problems with noise with varying dimensions for (a-c) Singular Broyden Problem (d-f) Flow in a Channel Problem (g-i) Swirling Flow Problem (j-l) Driven Cavity Problem}
	\label{fig:smoothnoisy2}
\end{figure*}

\begin{figure*}[t]
	\centering
	\includegraphics[scale=0.8]{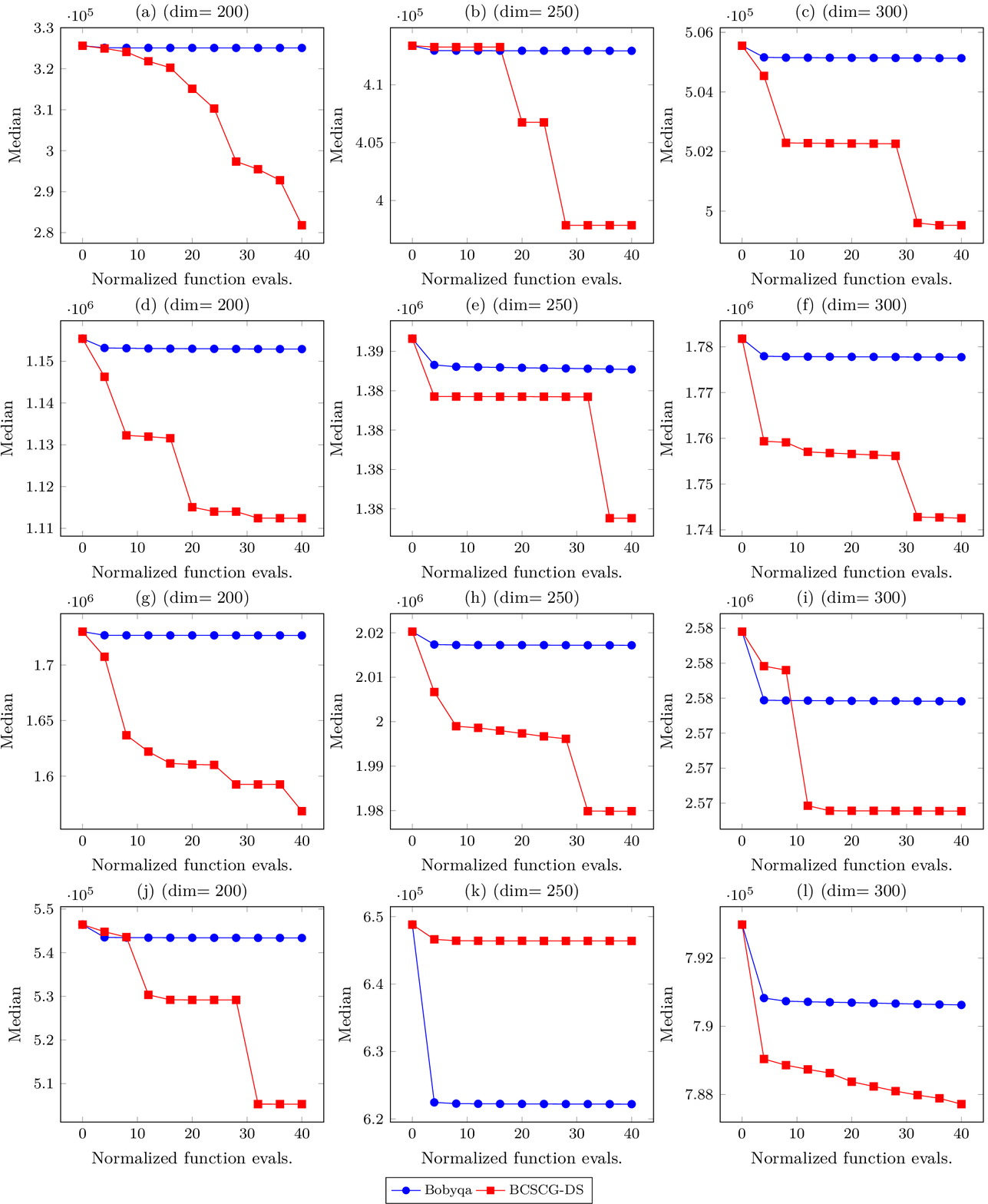}
	\caption{Progress curves for piecewise-smooth problems with noise with varying dimensions for (a-c) Generalized Broyden Tridiagonal Problem (d-f) Chained and Modified problem HS53 (g-i) Attracting-Repelling Problem (j-l) Modified Countercurrent Reactors Problem-2}
	\label{fig:piecewiseSmoothnoisy1}
\end{figure*}

\begin{figure*}[t]
	\centering
	\includegraphics[scale=0.8]{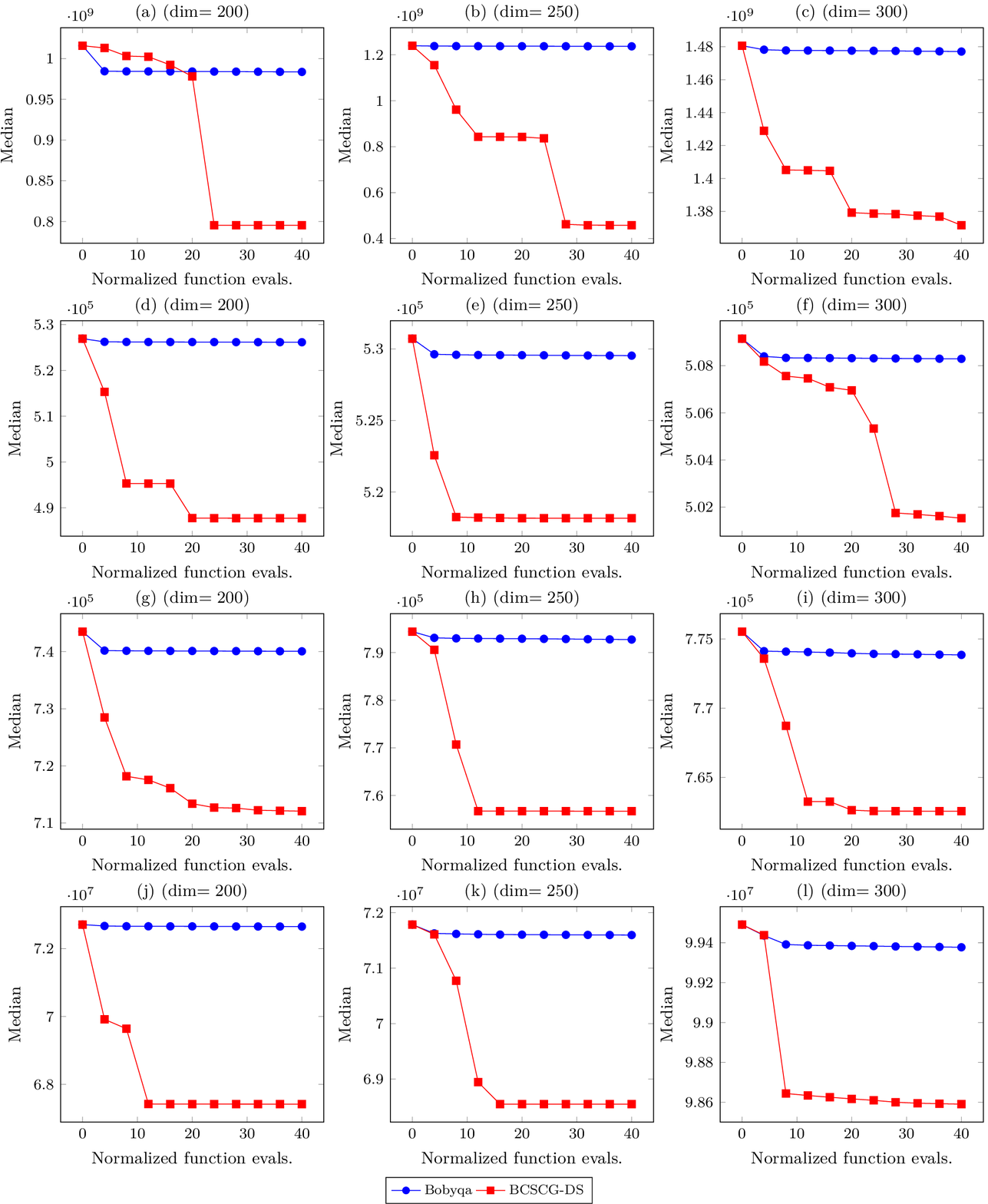}
	\caption{Progress curves for piecewise-smooth problems with noise with varying dimensions for (a-c) Singular Broyden Problem (d-f) Flow in a Channel Problem (g-i) Swirling Flow Problem (j-l) Driven Cavity Problem}
	\label{fig:piecewiseSmoothnoisy2}
\end{figure*}

 For quantitative comparison between the proposed method and Bobyqa method, we also show the progress curves \cite{regis2013combining} depicting the local optimum value versus the number of normalized function evaluations.
 For a given function, the local minimum function value is computed for a specific value of the starting point or initial guess and at a particular value of normalized function evaluations.
For better statistical relevance, the aforesaid computation is performed for ten different starting points and median value of the multiple minima obtained over ten trials is calculated.
This process is performed for each solver by varying the number of normalized function evaluations.
Subsequently, the median optimum function value (y-axis, log scale) is plotted with respect to normalized function evaluations (x-axis) for the given function.
The resulting plots for the two solvers obtained by considering eight noisy smooth and piecewise-smooth functions, with three different dimensions (200, 250, and 300) for each function, are shown in figures \ref{fig:smoothnoisy1} to \ref{fig:piecewiseSmoothnoisy2}.

Figures \ref{fig:smoothnoisy1} and \ref{fig:smoothnoisy2} illustrate the progress curves for 24 (eight functions, three dimensions) smooth problems.
To gain better understanding about the plots, consider figure \ref{fig:smoothnoisy1}(a) where the median of the optimum function values computed over ten trials versus normalized function evaluations is shown for Generalized Broyden Tridiagonal function with a dimension of magnitude 200.
As the number of normalized function evaluations increase, the median optima values obtained by the proposed BCSCG-DS method improve markedly, characterized by lowering values on the y-axis.
The median of function values evaluated at different starting points for both the solvers is approximately $9.438\times10^8$.
 The performance of different solvers is as follows:
 \begin{enumerate}
	 \item \textbf{Bobyqa}:  After 4 normalized function evaluations, the median of function values is approximately $9.407\times10^8$. 
		 It converges to the approximately $9.403\times10^8$ after 40 normalized function evaluations. 
	 \item \textbf{BCSCG-DS}:  After 4 normalized function evaluations, the median of function values is $9.100\times10^8$. 
		 It converges to approximately $3.412\times10^8$ after 40 normalized function evaluations.
 \end{enumerate}

The plots corresponding to dimensions of magnitude 250 and 300 for the Generalized Broyden Tridiagonal function are shown in figures \ref{fig:smoothnoisy1}(b) and \ref{fig:smoothnoisy1}(c). 
Similarly, progress curves are shown in figures \ref{fig:smoothnoisy1}(d-f) for Chained and Modified problem HS53, figures \ref{fig:smoothnoisy1}(g-i) for Attracting-Repelling Problem, and figures \ref{fig:smoothnoisy1}(j-l) for Modified Countercurrent Reactors Problem-2. 
The progress curve plots corresponding to dimensions of magnitude 200, 250 and 300 are shown for Singular Broyden Problem in figures \ref{fig:smoothnoisy2}(a-c), Flow Channel Problem in figures \ref{fig:smoothnoisy2}(d-f), Swirling Flow Problem in figures \ref{fig:smoothnoisy2}(g-i)  and Driven Cavity Problem in figures \ref{fig:smoothnoisy2}(j-l). 
Similarly, figures \ref{fig:piecewiseSmoothnoisy1} and \ref{fig:piecewiseSmoothnoisy2} illustrate the progress curves for 24 piecewise-smooth problems.

For further insights, tabular results are also presented which enable quantitative analysis of the performance of two solvers.
	Tables \ref{tab:smoothBobyqaCSCG} and \ref{tab:pwsmoothBobyqaCSCG} show the median (of 10 instances) of optimal function values attained by different solvers for smooth and piecewise-smooth problems.
Results are shown over three columns $200,250$ and $300$ which refer to the dimensions considered for experimentation.
For each dimension, the first two subcolumns refer to the median (of 10 instances) of optimal solutions found by Bobyqa and BCSCG-DS after $40$ normalized function evaluations while the third subcolumn refer to the median of function values of 10 different starting points.
	From the progress curves and tabular results, we can infer that the proposed method offers comparatively better optima values in contrast to Bobyqa.

\subsection{Computational Time}
We now analyse the computational time taken by the two solvers.
We compared our approach with Bobyqa on 4 different noisy smooth problems over 9 different dimensions varying from 100 to 500, summing up to total 36 instances.
We used 10 different starting points for each instance.
In order to reduce any ambiguity attributed to computational noise, we repeated the experiment five times.
We used gcc compiler (version 4.9.2) for both the solvers on Intel Xeon(R) E5-2670 processor with GNU Linux Operating System.
Same test codes \cite{lukvsan2018sparse} were used for function evaluations for both the solvers.

\begin{table}[h]
	\begin{center}
		\begin{tabular}{|p{5cm}|p{2cm}|p{2cm}|p{2cm}|p{2cm}|}
			\hline
			Problem&Dimension&Bobyqa Optimal Value&BCSCG-DS Optimal Value&Initial Function Value\\
			\hline
			\multirow{9}{5cm}{Chained Rosenbrock Function}
			&100&8.822e+05&6.002e+02&8.835e+05\\
			&150&1.267e+06&9.064e+05&1.269e+06\\
			&200&1.692e+06&1.642e+06&1.695e+06\\
			&250&2.041e+06&2.007e+06&2.041e+06\\
			&300&2.454e+06&2.333e+06&2.457e+06\\
			&350&2.946e+06&2.931e+06&2.948e+06\\
			&400&3.280e+06&3.273e+06&3.284e+06\\
			&450&3.800e+06&3.801e+06&3.803e+06\\
			&500&4.148e+06&4.136e+06&4.151e+06\\\hline
			\multirow{9}{5cm}{Chained and Modified HS47}
			&100&1.672e+06&1.284e+09&2.703e+10\\
			&150&3.446e+10&3.185e+10&5.198e+10\\
			&200&3.600e+10&5.654e+10&6.801e+10\\
			&250&6.372e+10&3.744e+10&7.984e+10\\
			&300&1.024e+11&8.663e+10&1.054e+11\\
			&350&1.224e+11&1.019e+11&1.236e+11\\
			&400&1.392e+11&1.254e+11&1.413e+11\\
			&450&1.477e+11&1.321e+11&1.502e+11\\
			&500&1.588e+11&1.544e+11&1.604e+11\\\hline
			\multirow{9}{5cm}{Chained and Modified HS48}
			&100&7.010e+05&3.434e+05&7.028e+05\\
			&150&1.083e+06&1.054e+06&1.084e+06\\
			&200&1.561e+06&1.450e+06&1.562e+06\\
			&250&1.799e+06&1.769e+06&1.802e+06\\
			&300&2.224e+06&2.210e+06&2.228e+06\\
			&350&2.646e+06&2.605e+06&2.650e+06\\
			&400&2.843e+06&2.827e+06&2.847e+06\\
			&450&3.293e+06&3.293e+06&3.297e+06\\
			&500&3.660e+06&3.661e+06&3.664e+06\\\hline
			\multirow{9}{5cm}{Modified Countercurrent Reactors Problem-1}
			&100&2.410e+05&1.395e+05&2.415e+05\\
			&150&3.679e+05&3.319e+05&3.685e+05\\
			&200&5.046e+05&4.629e+05&5.056e+05\\
			&250&6.059e+05&5.910e+05&6.069e+05\\
			&300&7.547e+05&7.546e+05&7.554e+05\\
			&350&8.637e+05&8.483e+05&8.650e+05\\
			&400&9.979e+05&9.974e+05&9.989e+05\\
			&450&1.114e+06&1.115e+06&1.115e+06\\
			&500&1.226e+06&1.225e+06&1.228e+06\\\hline
	\end{tabular}
\end{center}
\caption{Test Results for Problems with Varying Dimensions} \label{tab:timeOptimValues}
\end{table}

In table \ref{tab:timeOptimValues}, a comparison between the solutions attained by the two solvers after 40 normalized function evaluations is shown.
The first and second column represent the function name and its corresponding dimension.
The columns under the headers "Bobyqa Value" and "BCSCG-DS Value" represent the median of solution values obtained from 10 different starting points.
The last column with header "First Value" represents the median of initial function value evaluated at the ten different starting points.

In figure \ref{fig:timePlot}, a comparison of total computational time taken in seconds by both solvers is shown.
We observe that at dimensions close to 100, both solvers take nearly same time.
However, as dimensions increase, the computational burden of Bobyqa solver increases significantly faster than that of the proposed solver.  
Notably, BCSCG-DS becomes approximately five times faster than Bobyqa when the dimension of the problem increases to $500$.

From the table \ref{tab:timeOptimValues} and figure \ref{fig:timePlot}, it is clear that BCSCG-DS is competitive with respect to Bobyqa in terms of both solution quality and computational time.
\begin{figure*}[h]
	\centering
	\includegraphics[scale=1]{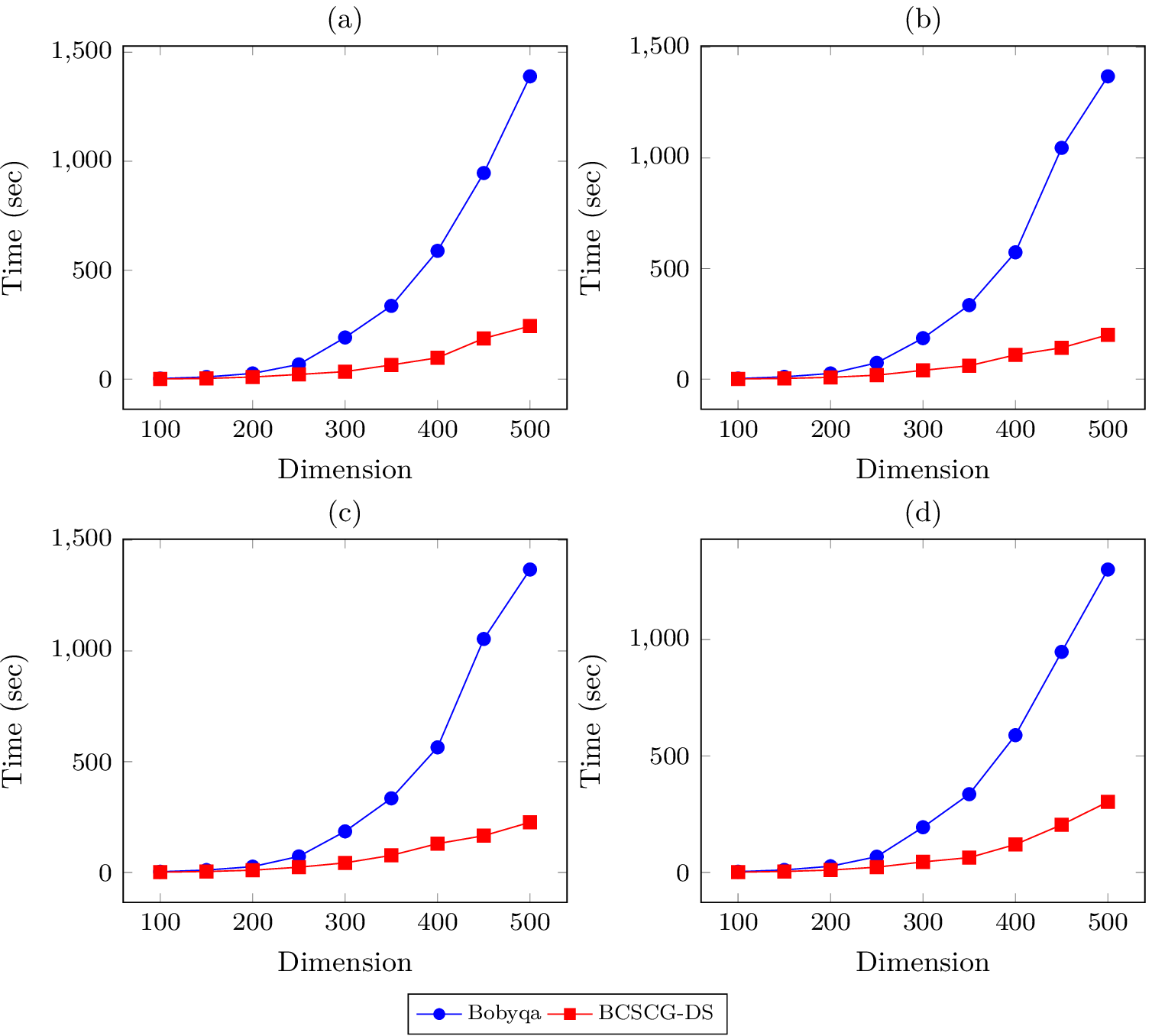}
	\caption{Computational time taken by Bobyqa and BCSCG-DS for problems (a)Chained Rosenbrock (b)Chained and Modified HS47 (c)Chained and Modified HS48 (d)Modified Countercurrent Reactors Problem-1}
	\label{fig:timePlot}
\end{figure*}

\section{Conclusion}\label{Conclusion}
In this paper, we proposed a new approach for solving high dimension box constrained derivative free optimization problems with noise.
We integrated scaled conjugate gradient algorithm with direct search approach whose performance is further enhanced by inclusion of quadratic models into the framework. 
The computational results clearly demonstrate its effectiveness and competitive performance when compared to standard solver.
Importantly, our approach offers three distinct advantages:
\begin{enumerate}
	\item Guaranteed convergence to local optimum.
	\item Extensibility towards larger dimensions for high dimensional optimization.
	\item Low computational time when compared to other solvers.
\end{enumerate}

In addition to suitability for high dimensional derivative free optimization, the proposed approach is also effective for solving noisy piecewise-smooth problems, as is evident from the extensive and rigorous numerical experiments performed.
This class of problems is conventionally considered to be more challenging than its smooth counterpart because of differentiability issues.
Consequently, our algorithm BCSCG-DS holds promising potential for solving the high dimensional box constrained derivative free optimization problems.

\clearpage
\begin{table}[h]
	\begin{center}
		\begin{tabular}{|c|l|}
			\hline
			S.No.&Function Name\\
			\hline
			1&Chained Rosenbrock function\\
			2&Chained Wood function\\
			3&Chained Powell singular function\\
			4&Chained Cragg and Levy function\\
			5&Generalized Broyden tridiagonal function\\
			6&Generalized Broyden banded function\\
			7&Chained Freudenstein and Roth function\\
			8&Wright and Holt zero residual problem\\
			9&Toint quadratic merging problem\\
			10&Chained serpentine function\\
			11&Chained and modified problem HS47\\
			12&Chained and modified problem HS48\\
			13&Sparse signomial function\\
			14&Sparse trigonometric function\\
			15&Countercurrent reactors problem 1(modified)\\
			16&Tridiagonal system\\
			17&Structured Jacobian problem\\
			18&Modified discrete boundary value problem\\
			19&Chained and modified problem HS53\\
			20&Attracting-Repelling problem\\
			21&Countercurrent reactors problem 2(modified)\\
			22&Trigonometric system\\
			23&Trigonometric - exponential system (trigexp 1)\\
			24&Singular Broyden problem\\
			25&Five-diagonal system\\
			26&Seven-diagonal system\\
			27&Extended Freudenstein and Roth function\\
			28&Broyden tridiagonal problem\\
			29&Extended Powell badly scaled function\\
			30&Extended Wood problem\\
			31&Tridiagonal exponential problem\\
			32&Discrete boundary value problem\\
			33&Brent problem\\
			34&Flow in a channel\\
			35&Swirling flow\\
			36&Bratu problem\\
			37&Poisson problem 1\\
			38&Poisson problem 2\\
			39&Porous medium problem\\
			40&Convection-difussion problem\\
			41&Nonlinear biharmonic problem\\
			42&Driven cavity problem\\
			43&Problem 2.47 of \cite{lukvsan2018sparse}\\
			44&Problem 2.48 of \cite{lukvsan2018sparse}\\
			45&Problem 2.49 of \cite{lukvsan2018sparse}\\
			46&Problem 2.50 of \cite{lukvsan2018sparse}\\
			47&Problem 2.51 of \cite{lukvsan2018sparse}\\
			48&Problem 2.52 of \cite{lukvsan2018sparse}\\
			49&Problem 2.53 of \cite{lukvsan2018sparse}\\
			50&Problem 2.54 of \cite{lukvsan2018sparse}\\
			51&Broyden banded function\\
			52&Ascher and Russel boundary value problem\\
			53&Potra and Rheinboldt boundary value problem\\
			54&Modified Bratu problem\\
			55&Nonlinear Dirichlet problem\\
			\hline
		\end{tabular}
	\end{center}
	\caption{Test Problems} \label{tab:testProblems}
\end{table}
\begin{table}[h]
	\begin{center}
		\begin{tabular}{|p{3mm}|p{1.3cm}|p{1.3cm}|p{1.4cm}|p{1.3cm}|p{1.3cm}|p{1.4cm}|p{1.3cm}|p{1.3cm}|p{1.4cm}|}
			\hline
			\multicolumn{1}{|c|}{\textbf{}}
			&
			\multicolumn{3}{|c|}{\textbf{Dimension=200}}
			&
			\multicolumn{3}{|c|}{\textbf{Dimension=250}}
			&
			\multicolumn{3}{|c|}{\textbf{Dimension=300}}\\\hline
			S No.&Bobyqa Optimal Value&BCSCG-DS Optimal Value&Initial Function Value&Bobyqa Optimal Value&BCSCG-DS Optimal Value&Initial Function Value&Bobyqa Optimal Value&BCSCG-DS Optimal Value&Initial Function Value\\
			\hline
			1&2.34e+10&2.01e+06&2.41e+10&3.17e+10&2.43e+10&3.19e+10&3.72e+10&3.62e+10&3.73e+10\\
			2&1.87e+10&1.21e+10&2.46e+10&2.87e+10&2.54e+10&2.89e+10&3.62e+10&3.60e+10&3.77e+10\\
			3&8.71e+09&6.80e+09&9.43e+09&1.39e+10&1.22e+10&1.41e+10&1.64e+10&1.50e+10&1.64e+10\\
			4&3.20e+73&4.09e+48&3.89e+85&9.96e+72&1.07e+45&7.21e+84&1.00e+74&1.29e+62&1.08e+86\\
			5&9.40e+08&3.41e+08&9.44e+08&1.23e+09&1.07e+09&1.23e+09&1.50e+09&1.43e+09&1.50e+09\\
			6&1.13e+13&3.60e+08&1.14e+13&1.31e+13&9.63e+08&1.48e+13&1.64e+13&1.46e+13&1.65e+13\\
			7&8.20e+11&9.05e+06&8.61e+11&1.09e+12&2.15e+07&1.12e+12&1.24e+12&1.16e+12&1.27e+12\\
			8&6.10e+60&1.24e+59&6.28e+84&3.36e+66&1.55e+73&1.15e+85&7.95e+61&3.65e+79&1.85e+85\\
			9&2.05e+13&2.46e+13&4.01e+13&1.50e+13&2.28e+13&5.52e+13&1.27e+13&5.04e+13&7.94e+13\\
			10&1.67e+07&1.46e+07&1.67e+07&2.04e+07&2.04e+07&2.04e+07&2.46e+07&2.45e+07&2.46e+07\\
			11&1.23e+13&4.23e+19&3.64e+20&1.92e+14&1.46e+20&4.68e+20&1.89e+16&2.40e+20&4.35e+20\\
			12&1.87e+10&1.48e+10&1.88e+10&2.16e+10&2.07e+10&2.21e+10&2.61e+10&2.44e+10&2.62e+10\\
			13&4.52e+26&7.50e+29&7.56e+32&3.90e+26&4.53e+31&1.08e+33&2.37e+26&1.32e+32&1.34e+33\\
			14&1.08e+04&4.14e+06&6.89e+06&1.31e+04&4.98e+06&8.18e+06&1.94e+04&8.41e+06&9.88e+06\\
			15&2.19e+09&1.97e+09&2.20e+09&2.74e+09&2.66e+09&2.75e+09&3.13e+09&3.10e+09&3.15e+09\\
			16&2.67e+13&6.24e+06&2.88e+13&2.95e+13&1.37e+13&3.31e+13&4.21e+13&1.15e+13&4.23e+13\\
			17&9.79e+08&9.63e+08&9.84e+08&1.24e+09&1.16e+09&1.27e+09&1.48e+09&1.26e+09&1.49e+09\\
			18&1.09e+06&9.25e+05&1.09e+06&1.10e+06&1.09e+06&1.10e+06&1.44e+06&1.43e+06&1.44e+06\\
			19&1.58e+10&1.18e+10&1.60e+10&2.00e+10&1.63e+10&2.07e+10&2.36e+10&2.29e+10&2.36e+10\\
			20&2.41e+10&3.10e+09&2.42e+10&3.10e+10&2.82e+10&3.11e+10&3.61e+10&3.51e+10&3.62e+10\\
			21&2.60e+09&2.17e+09&2.79e+09&2.99e+09&2.81e+09&3.17e+09&3.46e+09&3.28e+09&3.80e+09\\
			22&1.17e+01&5.59e+04&1.31e+05&1.97e+01&1.21e+05&2.73e+05&3.91e+01&2.52e+05&4.50e+05\\
			23&5.05e+71&3.09e+40&2.20e+83&4.54e+70&7.27e+56&1.22e+85&7.77e+70&1.37e+63&6.13e+84\\
			24&1.18e+16&3.37e+10&1.22e+16&1.48e+16&4.93e+12&1.67e+16&1.94e+16&1.48e+15&2.04e+16\\
			25&2.69e+13&2.71e+10&2.78e+13&3.87e+13&3.70e+13&3.89e+13&4.37e+13&4.21e+13&4.39e+13\\
			26&2.83e+13&7.75e+07&2.85e+13&3.38e+13&1.10e+09&3.56e+13&4.91e+13&4.46e+13&4.94e+13\\
			27&2.74e+11&2.59e+11&4.37e+11&4.51e+11&4.13e+11&5.20e+11&6.38e+11&5.69e+11&6.48e+11\\
			28&6.40e+07&5.77e+07&6.41e+07&8.10e+07&8.11e+07&8.11e+07&8.88e+07&8.76e+07&9.17e+07\\
			29&1.29e+25&2.94e+28&3.45e+43&1.58e+25&2.50e+35&2.78e+43&9.07e+24&7.79e+34&2.99e+43\\
			30&7.48e+15&5.53e+15&7.87e+15&9.35e+15&7.05e+15&1.00e+16&1.19e+16&1.03e+16&1.20e+16\\
			31&1.62e+05&1.66e+05&1.73e+05&1.99e+05&2.04e+05&2.11e+05&2.34e+05&2.41e+05&2.49e+05\\
			32&9.63e+05&9.32e+05&9.65e+05&1.25e+06&1.25e+06&1.25e+06&1.53e+06&1.47e+06&1.53e+06\\
			33&1.00e+10&7.59e+09&1.01e+10&1.33e+10&1.26e+10&1.33e+10&1.56e+10&1.56e+10&1.63e+10\\
			34&2.40e+09&2.15e+09&2.42e+09&1.74e+09&1.25e+09&1.85e+09&1.76e+09&1.47e+09&1.76e+09\\
			35&6.69e+09&5.42e+09&6.73e+09&4.21e+09&3.90e+09&4.22e+09&3.91e+09&3.19e+09&3.92e+09\\
			36&3.84e+20&6.08e+05&2.70e+40&6.25e+20&7.79e+07&2.71e+40&3.35e+21&9.20e+08&1.89e+40\\
			37&1.39e+07&1.39e+07&1.40e+07&1.26e+07&1.23e+07&1.26e+07&1.21e+07&1.13e+07&1.22e+07\\
			38&3.17e+06&2.98e+06&3.18e+06&3.38e+06&3.14e+06&3.39e+06&4.78e+06&4.73e+06&4.79e+06\\
			39&8.39e+12&4.79e+12&9.16e+12&9.43e+12&8.83e+12&9.51e+12&9.91e+12&8.36e+12&1.05e+13\\
			40&7.84e+08&7.29e+08&8.21e+08&8.96e+08&8.14e+08&8.99e+08&8.27e+08&7.65e+08&8.29e+08\\
			41&1.04e+08&1.00e+08&1.04e+08&1.13e+08&1.05e+08&1.14e+08&1.70e+08&1.70e+08&1.70e+08\\
			42&4.06e+13&1.97e+13&4.16e+13&4.90e+13&4.40e+13&4.92e+13&6.36e+13&5.68e+13&6.45e+13\\
			43&8.89e+08&8.88e+08&8.90e+08&4.58e+08&4.58e+08&4.59e+08&2.67e+08&2.67e+08&2.67e+08\\
			44&1.17e+25&4.97e+17&1.16e+44&5.00e+24&3.27e+39&1.22e+44&3.99e+24&1.61e+21&6.15e+43\\
			45&2.50e+06&1.97e+06&2.50e+06&3.47e+06&2.98e+06&3.48e+06&3.96e+06&3.70e+06&3.97e+06\\
			46&1.85e+15&5.75e+07&1.25e+34&5.19e+15&5.89e+11&5.75e+33&2.20e+16&3.50e+22&4.74e+33\\
			47&4.48e+06&3.26e+06&4.50e+06&6.09e+06&6.10e+06&6.11e+06&7.03e+06&7.03e+06&7.05e+06\\
			48&4.48e+11&2.14e+07&4.51e+11&5.29e+11&2.87e+07&5.41e+11&6.32e+11&1.87e+08&6.40e+11\\
			49&1.67e+05&1.60e+05&1.67e+05&2.08e+05&2.05e+05&2.09e+05&2.48e+05&2.48e+05&2.49e+05\\
			50&9.41e+05&8.31e+05&9.43e+05&1.18e+06&1.13e+06&1.18e+06&1.47e+06&1.47e+06&1.47e+06\\
			51&1.01e+13&1.15e+09&1.13e+13&1.27e+13&8.52e+09&1.37e+13&1.69e+13&1.58e+13&1.70e+13\\
			52&1.04e+06&1.02e+06&1.04e+06&1.27e+06&1.25e+06&1.28e+06&1.57e+06&1.51e+06&1.57e+06\\
			53&1.00e+06&8.83e+05&1.01e+06&1.23e+06&1.19e+06&1.23e+06&1.51e+06&1.48e+06&1.51e+06\\
			54&2.16e+20&1.53e+11&3.54e+38&1.51e+19&1.67e+15&3.22e+38&9.57e+18&1.99e+26&3.84e+38\\
			55&3.23e+06&2.92e+06&3.24e+06&3.75e+06&3.57e+06&3.75e+06&4.93e+06&4.58e+06&4.94e+06\\
			\hline
		\end{tabular}
	\end{center}
	\vspace{-3mm}
	\caption{Test Results for High Dimension Smooth Problems with Noise} \label{tab:smoothBobyqaCSCG}
\end{table}
\begin{table}[h]
	\begin{center}
\begin{tabular}{|p{3mm}|p{1.3cm}|p{1.3cm}|p{1.4cm}|p{1.3cm}|p{1.3cm}|p{1.4cm}|p{1.3cm}|p{1.3cm}|p{1.4cm}|}
	\hline
	\multicolumn{1}{|c|}{}
	&
	\multicolumn{3}{|c|}{\textbf{Dimension=200}}
	&
	\multicolumn{3}{|c|}{\textbf{Dimension=250}}
	&
	\multicolumn{3}{|c|}{\textbf{Dimension=300}}\\\hline
	S No.&Bobyqa Optimal Value&BCSCG-DS Optimal Value&Initial Function Value&Bobyqa Optimal Value&BCSCG-DS Optimal Value&Initial Function Value&Bobyqa Optimal Value&BCSCG-DS Optimal Value&Initial Function Value\\
	\hline
     1&1.67e+06&1.61e+06&1.68e+06&2.08e+06&1.99e+06&2.08e+06&2.55e+06&2.55e+06&2.56e+06\\
     2&1.62e+06&1.56e+06&1.63e+06&2.05e+06&2.00e+06&2.05e+06&2.41e+06&2.36e+06&2.41e+06\\
     3&9.14e+05&8.21e+05&9.16e+05&1.19e+06&1.19e+06&1.19e+06&1.37e+06&1.33e+06&1.37e+06\\
     4&1.36e+22&8.48e+24&1.70e+43&9.96e+20&2.99e+29&8.78e+42&6.58e+20&9.56e+34&6.30e+42\\
     5&3.25e+05&2.82e+05&3.26e+05&4.13e+05&3.98e+05&4.13e+05&5.05e+05&5.00e+05&5.06e+05\\
     6&3.26e+07&1.16e+07&3.26e+07&3.98e+07&3.40e+07&3.99e+07&4.77e+07&4.53e+07&4.78e+07\\
     7&1.26e+07&1.08e+07&1.26e+07&1.53e+07&1.51e+07&1.53e+07&1.82e+07&1.76e+07&1.82e+07\\
     8&1.49e+28&2.66e+39&7.16e+42&1.95e+32&2.94e+34&6.28e+42&5.75e+28&2.29e+41&7.60e+42\\
     9&3.68e+07&3.14e+07&3.75e+07&4.51e+07&5.07e+07&5.21e+07&5.71e+07&5.10e+07&5.73e+07\\
    10&5.24e+04&5.23e+04&5.24e+04&6.76e+04&6.75e+04&6.78e+04&8.27e+04&8.16e+04&8.27e+04\\
    11&5.90e+10&3.46e+10&6.58e+10&8.73e+10&6.22e+10&9.48e+10&9.87e+10&9.57e+10&1.05e+11\\
    12&1.43e+06&1.40e+06&1.43e+06&1.82e+06&1.72e+06&1.83e+06&2.24e+06&2.24e+06&2.24e+06\\
    13&1.34e+15&6.55e+15&7.80e+16&6.44e+16&4.87e+16&1.04e+17&6.85e+16&5.30e+16&1.04e+17\\
    14&8.85e+02&3.32e+04&4.16e+04&1.60e+03&4.70e+04&5.21e+04&1.67e+03&5.26e+04&6.18e+04\\
    15&4.93e+05&4.55e+05&4.93e+05&6.13e+05&6.02e+05&6.13e+05&7.59e+05&7.41e+05&7.61e+05\\
    16&4.95e+07&6.67e+06&4.96e+07&6.22e+07&5.78e+07&6.30e+07&7.26e+07&6.96e+07&7.27e+07\\
    17&3.41e+05&3.18e+05&3.41e+05&4.19e+05&4.04e+05&4.20e+05&4.97e+05&4.97e+05&4.97e+05\\
    18&1.17e+04&1.16e+04&1.17e+04&1.43e+04&1.41e+04&1.44e+04&1.80e+04&1.79e+04&1.80e+04\\
    19&1.15e+06&1.11e+06&1.16e+06&1.39e+06&1.38e+06&1.39e+06&1.78e+06&1.74e+06&1.78e+06\\
    20&1.73e+06&1.57e+06&1.73e+06&2.02e+06&1.98e+06&2.02e+06&2.58e+06&2.57e+06&2.58e+06\\
    21&5.43e+05&5.05e+05&5.46e+05&6.22e+05&6.46e+05&6.49e+05&7.91e+05&7.88e+05&7.93e+05\\
    22&2.81e+03&2.69e+03&3.40e+03&1.04e+02&4.39e+03&5.56e+03&7.05e+03&7.17e+03&8.08e+03\\
    23&2.64e+21&6.19e+21&2.66e+42&4.75e+21&1.33e+33&3.51e+41&3.17e+23&1.61e+33&1.99e+43\\
    24&9.84e+08&7.95e+08&1.02e+09&1.24e+09&4.58e+08&1.24e+09&1.48e+09&1.37e+09&1.48e+09\\
    25&4.92e+07&4.73e+07&4.93e+07&6.38e+07&5.80e+07&6.40e+07&7.33e+07&7.22e+07&7.35e+07\\
    26&5.08e+07&4.33e+07&5.09e+07&6.09e+07&5.90e+07&6.10e+07&7.20e+07&6.87e+07&7.21e+07\\
    27&6.00e+06&5.57e+06&6.03e+06&7.95e+06&7.56e+06&7.97e+06&9.46e+06&9.68e+06&9.76e+06\\
    28&8.72e+04&8.22e+04&8.74e+04&1.02e+05&1.00e+05&1.02e+05&1.27e+05&1.26e+05&1.27e+05\\
    29&1.92e+10&3.81e+14&8.00e+21&4.83e+10&6.72e+18&1.29e+22&1.19e+12&2.26e+18&1.84e+22\\
    30&6.61e+08&6.08e+08&6.63e+08&7.40e+08&7.15e+08&7.42e+08&9.05e+08&8.88e+08&9.07e+08\\
    31&4.81e+03&4.91e+03&5.01e+03&5.85e+03&5.94e+03&6.08e+03&7.17e+03&7.32e+03&7.47e+03\\
    32&1.17e+04&1.16e+04&1.17e+04&1.38e+04&1.36e+04&1.38e+04&1.76e+04&1.75e+04&1.76e+04\\
    33&9.94e+05&9.40e+05&9.96e+05&1.22e+06&1.19e+06&1.22e+06&1.56e+06&1.54e+06&1.56e+06\\
    34&5.26e+05&4.88e+05&5.27e+05&5.30e+05&5.18e+05&5.31e+05&5.08e+05&5.02e+05&5.09e+05\\
    35&7.40e+05&7.12e+05&7.43e+05&7.93e+05&7.57e+05&7.94e+05&7.74e+05&7.63e+05&7.76e+05\\
    36&1.08e+09&6.20e+03&2.73e+20&4.14e+08&1.14e+04&2.78e+20&5.51e+09&1.77e+18&2.86e+20\\
    37&3.98e+04&3.65e+04&3.99e+04&4.14e+04&4.09e+04&4.15e+04&4.98e+04&4.97e+04&4.99e+04\\
    38&2.07e+04&2.01e+04&2.08e+04&2.48e+04&2.48e+04&2.49e+04&3.15e+04&3.15e+04&3.16e+04\\
    39&2.54e+07&2.00e+07&2.60e+07&2.80e+07&2.55e+07&2.81e+07&3.22e+07&3.20e+07&3.23e+07\\
    40&2.80e+05&2.56e+05&2.81e+05&3.15e+05&3.15e+05&3.15e+05&3.71e+05&3.71e+05&3.72e+05\\
    41&1.18e+05&1.15e+05&1.21e+05&1.37e+05&1.35e+05&1.37e+05&1.75e+05&1.74e+05&1.75e+05\\
    42&7.27e+07&6.74e+07&7.27e+07&7.16e+07&6.85e+07&7.18e+07&9.94e+07&9.86e+07&9.95e+07\\
    43&6.26e+04&6.25e+04&6.27e+04&5.49e+04&5.49e+04&5.49e+04&5.13e+04&5.11e+04&5.14e+04\\
    44&1.88e+10&3.96e+08&2.09e+22&1.07e+11&1.76e+06&2.09e+22&2.15e+11&2.30e+18&1.91e+22\\
    45&1.75e+04&1.73e+04&1.75e+04&2.24e+04&2.24e+04&2.25e+04&2.66e+04&2.66e+04&2.66e+04\\
    46&1.07e+06&2.07e+12&3.00e+17&6.11e+06&4.55e+15&2.09e+17&6.45e+06&5.84e+13&1.73e+17\\
    47&2.22e+04&2.09e+04&2.23e+04&2.85e+04&2.83e+04&2.86e+04&3.49e+04&3.48e+04&3.50e+04\\
    48&6.67e+06&9.23e+05&6.69e+06&8.15e+06&7.80e+06&8.17e+06&9.60e+06&9.50e+06&9.62e+06\\
    49&5.03e+03&4.99e+03&5.03e+03&6.19e+03&6.17e+03&6.19e+03&7.48e+03&7.44e+03&7.49e+03\\
    50&1.13e+04&1.13e+04&1.13e+04&1.47e+04&1.46e+04&1.47e+04&1.71e+04&1.70e+04&1.71e+04\\
    51&3.17e+07&1.39e+07&3.17e+07&3.87e+07&3.86e+07&3.89e+07&4.73e+07&4.73e+07&4.74e+07\\
    52&1.18e+04&1.18e+04&1.18e+04&1.48e+04&1.45e+04&1.48e+04&1.76e+04&1.76e+04&1.76e+04\\
    53&1.17e+04&1.17e+04&1.18e+04&1.41e+04&1.39e+04&1.41e+04&1.75e+04&1.75e+04&1.75e+04\\
    54&9.04e+07&1.98e+10&4.72e+19&1.86e+08&4.05e+15&5.60e+19&1.29e+09&4.18e+17&4.71e+19\\
    55&2.08e+04&2.01e+04&2.08e+04&2.37e+04&2.34e+04&2.37e+04&3.06e+04&3.06e+04&3.06e+04\\
    \hline
\end{tabular}
\end{center}
	\vspace{-3mm}
	\caption{Test Results for Piecewise-smooth Problems with Noise} \label{tab:pwsmoothBobyqaCSCG}
\end{table}
\clearpage


\begin{thebibliography}{36}
\providecommand{\natexlab}[1]{#1}
\providecommand{\url}[1]{\texttt{#1}}
\expandafter\ifx\csname urlstyle\endcsname\relax
  \providecommand{\doi}[1]{doi: #1}\else
  \providecommand{\doi}{doi: \begingroup \urlstyle{rm}\Url}\fi

\bibitem[Conn et~al.(2009)Conn, Scheinberg, and Vicente]{connBook}
Andrew~R Conn, Katya Scheinberg, and Luis~N Vicente.
\newblock \emph{Introduction to derivative-free optimization}, volume~8.
\newblock Siam, 2009.

\bibitem[Audet and Hare(2017)]{audet2017derivative}
Charles Audet and Warren Hare.
\newblock \emph{Derivative-free and blackbox optimization}.
\newblock Springer, 2017.

\bibitem[Clarke(1990)]{clarke1990optimization}
Frank~H Clarke.
\newblock \emph{Optimization and nonsmooth analysis}, volume~5.
\newblock Siam, 1990.

\bibitem[Audet and Dennis~Jr(2006)]{audet_dennis_jr_mesh}
Charles Audet and John~E Dennis~Jr.
\newblock Mesh adaptive direct search algorithms for constrained optimization.
\newblock \emph{SIAM Journal on optimization}, 17\penalty0 (1):\penalty0
  188--217, 2006.

\bibitem[Abramson et~al.(2009)Abramson, Audet, Dennis~Jr, and
  Digabel]{abramson_audet_ea_orthomads}
Mark~A Abramson, Charles Audet, John~E Dennis~Jr, and S{\'e}bastien~Le Digabel.
\newblock Orthomads: A deterministic mads instance with orthogonal directions.
\newblock \emph{SIAM Journal on Optimization}, 20\penalty0 (2):\penalty0
  948--966, 2009.

\bibitem[Vicente and Cust{\'o}dio(2012)]{vicente2012analysis}
Lu{\'\i}s~N Vicente and AL~Cust{\'o}dio.
\newblock Analysis of direct searches for discontinuous functions.
\newblock \emph{Mathematical programming}, 133\penalty0 (1-2):\penalty0
  299--325, 2012.

\bibitem[Fasano et~al.(2014)Fasano, Liuzzi, Lucidi, and
  Rinaldi]{fasano2014linesearch}
Giovanni Fasano, Giampaolo Liuzzi, Stefano Lucidi, and Francesco Rinaldi.
\newblock A linesearch-based derivative-free approach for nonsmooth constrained
  optimization.
\newblock \emph{SIAM Journal on Optimization}, 24\penalty0 (3):\penalty0
  959--992, 2014.

\bibitem[Gray and Kolda(2006)]{gray2006algorithm}
Genetha~A Gray and Tamara~G Kolda.
\newblock Algorithm 856: Appspack 4.0: Asynchronous parallel pattern search for
  derivative-free optimization.
\newblock \emph{ACM Transactions on Mathematical Software (TOMS)}, 32\penalty0
  (3):\penalty0 485--507, 2006.

\bibitem[{Le~Digabel}(2011)]{Le09b}
S.~{Le~Digabel}.
\newblock {Algorithm 909: NOMAD: Nonlinear Optimization with the MADS
  algorithm}.
\newblock \emph{{ACM} Transactions on Mathematical Software}, 37\penalty0
  (4):\penalty0 44:1--44:15, 2011.

\bibitem[Audet et~al.(2008{\natexlab{a}})Audet, Dennis~Jr, and
  Digabel]{audet2008parallel}
Charles Audet, John~E Dennis~Jr, and S{\'e}bastien~Le Digabel.
\newblock Parallel space decomposition of the mesh adaptive direct search
  algorithm.
\newblock \emph{SIAM Journal on Optimization}, 19\penalty0 (3):\penalty0
  1150--1170, 2008{\natexlab{a}}.

\bibitem[Cust{\'o}dio and Vicente(2007)]{custodio2007using}
Ana~Lu{\i}sa Cust{\'o}dio and Lu{\'\i}s~N Vicente.
\newblock Using sampling and simplex derivatives in pattern search methods.
\newblock \emph{SIAM Journal on Optimization}, 18\penalty0 (2):\penalty0
  537--555, 2007.

\bibitem[Cust{\'o}dio et~al.(2008)Cust{\'o}dio, Dennis~Jr, and
  Vicente]{custodio2008using}
AL~Cust{\'o}dio, John~E Dennis~Jr, and Lu{\'\i}s~Nunes Vicente.
\newblock Using simplex gradients of nonsmooth functions in direct search
  methods.
\newblock \emph{IMA Journal of Numerical Analysis}, 28\penalty0 (4):\penalty0
  770--784, 2008.

\bibitem[Powell(2006)]{powell2006newuoa}
Michael~J.D. Powell.
\newblock The newuoa software for unconstrained optimization without
  derivatives.
\newblock In \emph{Large-scale nonlinear optimization}, pages 255--297.
  Springer, 2006.

\bibitem[Wang and Zhu(2016)]{wang2016conjugate}
Jueyu Wang and Detong Zhu.
\newblock Conjugate gradient path method without line search technique for
  derivative-free unconstrained optimization.
\newblock \emph{Numerical Algorithms}, 73\penalty0 (4):\penalty0 957--983,
  2016.

\bibitem[Jones et~al.(1998)Jones, Schonlau, and Welch]{jones1998efficient}
Donald~R Jones, Matthias Schonlau, and William~J Welch.
\newblock Efficient global optimization of expensive black-box functions.
\newblock \emph{Journal of Global optimization}, 13\penalty0 (4):\penalty0
  455--492, 1998.

\bibitem[Gramacy and {Le~Digabel}(2015)]{gramacy_le_digabel_mesh}
R.B. Gramacy and S.~{Le~Digabel}.
\newblock {The mesh adaptive direct search algorithm with treed Gaussian
  process surrogates}.
\newblock \emph{Pacific Journal of Optimization}, 11\penalty0 (3):\penalty0
  419--447, 2015.
\newblock URL \url{http://www.ybook.co.jp/online2/pjov11-3.html}.

\bibitem[Wild et~al.(2008)Wild, Regis, and Shoemaker]{wild_regis_ea_orbit}
Stefan~M Wild, Rommel~G Regis, and Christine~A Shoemaker.
\newblock Orbit: Optimization by radial basis function interpolation in
  trust-regions.
\newblock \emph{SIAM Journal on Scientific Computing}, 30\penalty0
  (6):\penalty0 3197--3219, 2008.

\bibitem[Regis and Shoemaker(2013)]{regis2013combining}
Rommel~G Regis and Christine~A Shoemaker.
\newblock Combining radial basis function surrogates and dynamic coordinate
  search in high-dimensional expensive black-box optimization.
\newblock \emph{Engineering Optimization}, 45\penalty0 (5):\penalty0 529--555,
  2013.

\bibitem[Mor{\'e} and Wild(2009)]{more_wild_benchmarking}
Jorge~J Mor{\'e} and Stefan~M Wild.
\newblock Benchmarking derivative-free optimization algorithms.
\newblock \emph{SIAM Journal on Optimization}, 20\penalty0 (1):\penalty0
  172--191, 2009.

\bibitem[Cust{\'o}dio et~al.(2010)Cust{\'o}dio, Rocha, and
  Vicente]{custodio2010incorporating}
Ana~Lu{\'\i}sa Cust{\'o}dio, Humberto Rocha, and Lu{\'\i}s~N Vicente.
\newblock Incorporating minimum frobenius norm models in direct search.
\newblock \emph{Computational Optimization and Applications}, 46\penalty0
  (2):\penalty0 265--278, 2010.

\bibitem[Conn and Le~Digabel(2013)]{conn_le_digabel_use}
Andrew~R Conn and S{\'e}bastien Le~Digabel.
\newblock Use of quadratic models with mesh-adaptive direct search for
  constrained black box optimization.
\newblock \emph{Optimization Methods and Software}, 28\penalty0 (1):\penalty0
  139--158, 2013.

\bibitem[Amaioua et~al.(2018)Amaioua, Audet, Conn, and
  Le~Digabel]{amaioua2018efficient}
Nadir Amaioua, Charles Audet, Andrew~R Conn, and S{\'e}bastien Le~Digabel.
\newblock Efficient solution of quadratically constrained quadratic subproblems
  within the mesh adaptive direct search algorithm.
\newblock \emph{European Journal of Operational Research}, 268\penalty0
  (1):\penalty0 13--24, 2018.

\bibitem[Audet et~al.(2008{\natexlab{b}})Audet, B{\'e}chard, and
  Le~Digabel]{audet_bechard_ea_nonsmooth}
Charles Audet, Vincent B{\'e}chard, and S{\'e}bastien Le~Digabel.
\newblock Nonsmooth optimization through mesh adaptive direct search and
  variable neighborhood search.
\newblock \emph{Journal of Global Optimization}, 41\penalty0 (2):\penalty0
  299--318, 2008{\natexlab{b}}.

\bibitem[Alberto et~al.(2004)Alberto, Nogueira, Rocha, and
  Vicente]{alberto_nogueira_ea_pattern}
Pedro Alberto, Fernando Nogueira, Humberto Rocha, and Lu{\'\i}s~N Vicente.
\newblock Pattern search methods for user-provided points: Application to
  molecular geometry problems.
\newblock \emph{SIAM Journal on Optimization}, 14\penalty0 (4):\penalty0
  1216--1236, 2004.

\bibitem[Halton(1960)]{halton1960efficiency}
John~H Halton.
\newblock On the efficiency of certain quasi-random sequences of points in
  evaluating multi-dimensional integrals.
\newblock \emph{Numerische Mathematik}, 2\penalty0 (1):\penalty0 84--90, 1960.

\bibitem[Sobol(1976)]{sobol1976uniformly}
Ilya~M Sobol.
\newblock Uniformly distributed sequences with an additional uniform property.
\newblock \emph{USSR Computational Mathematics and Mathematical Physics},
  16\penalty0 (5):\penalty0 236--242, 1976.

\bibitem[Edelsbrunner and Grayson(2000)]{edelsbrunner2000edgewise}
Herbert Edelsbrunner and Daniel~R Grayson.
\newblock Edgewise subdivision of a simplex.
\newblock \emph{Discrete \& Computational Geometry}, 24\penalty0 (4):\penalty0
  707--719, 2000.

\bibitem[Golub and Van~Loan(2012)]{golub2012matrix}
Gene~H Golub and Charles~F Van~Loan.
\newblock \emph{Matrix computations}, volume~3.
\newblock JHU Press, 2012.

\bibitem[Bortz and Kelley(1998)]{Bortz98thesimplex}
D.~M. Bortz and C.~T. Kelley.
\newblock The simplex gradient and noisy optimization problems.
\newblock In \emph{in Computational Methods in Optimal Design and Control},
  pages 77--90. Birkhäuser, 1998.

\bibitem[Hager and Zhang(2006)]{hager_zhang_survey}
William~W Hager and Hongchao Zhang.
\newblock A survey of nonlinear conjugate gradient methods.
\newblock \emph{Pacific journal of Optimization}, 2\penalty0 (1):\penalty0
  35--58, 2006.

\bibitem[Andrei(2008)]{andrei2008unconstrained}
Neculai Andrei.
\newblock An unconstrained optimization test functions collection.
\newblock \emph{Adv. Model. Optim}, 10\penalty0 (1):\penalty0 147--161, 2008.

\bibitem[Birgin and Mart{\'\i}nez(2001)]{birgin2001spectral}
Ernesto~G Birgin and Jos{\'e}~Mario Mart{\'\i}nez.
\newblock A spectral conjugate gradient method for unconstrained optimization.
\newblock \emph{Applied Mathematics and optimization}, 43\penalty0
  (2):\penalty0 117--128, 2001.

\bibitem[Powell(2009)]{powell2009bobyqa}
M.J.D. Powell.
\newblock The bobyqa algorithm for bound constrained optimization without
  derivatives (technical report no. damtp 2009/na06).
\newblock Technical report, Technical report, Cambridge: University of
  Cambridge, Department of Applied Mathematics and Theoretical Physics, Centre
  for Mathematical Sciences, 2009.

\bibitem[Anderson et~al.(1999)Anderson, Bai, Bischof, Blackford, Demmel,
  Dongarra, Du~Croz, Greenbaum, Hammarling, McKenney, and Sorensen]{laug}
E.~Anderson, Z.~Bai, C.~Bischof, S.~Blackford, J.~Demmel, J.~Dongarra,
  J.~Du~Croz, A.~Greenbaum, S.~Hammarling, A.~McKenney, and D.~Sorensen.
\newblock \emph{{LAPACK} Users' Guide}.
\newblock Society for Industrial and Applied Mathematics, Philadelphia, PA,
  third edition, 1999.
\newblock ISBN 0-89871-447-8 (paperback).

\bibitem[Luk{\v{s}}an et~al.(2018)Luk{\v{s}}an, Matonoha, and
  Vlcek]{lukvsan2018sparse}
Ladislav Luk{\v{s}}an, Ctirad Matonoha, and Jan Vlcek.
\newblock Sparse test problems for nonlinear least squares.
\newblock \emph{Prague: ICS CAS, 2018. 20 s. Technical Report, V-1258}, 2018.

\bibitem[Dolan and Mor{\'e}(2002)]{dolan2002benchmarking}
Elizabeth~D Dolan and Jorge~J Mor{\'e}.
\newblock Benchmarking optimization software with performance profiles.
\newblock \emph{Mathematical programming}, 91\penalty0 (2):\penalty0 201--213,
  2002.

\end{thebibliography}
\end{document}